\documentclass[reqno]{amsart}
\usepackage{amsmath,amsfonts,amssymb,amsthm}
\usepackage{enumitem}
\usepackage[usenames,dvipsnames]{xcolor}
\usepackage[colorlinks=true]{hyperref} 
\usepackage{verbatim}
\usepackage{graphicx}
\hypersetup{linkcolor=Violet,citecolor=PineGreen}
\newtheorem{thm}{Theorem}[section]
\newtheorem{lem}[thm]{Lemma}
\newtheorem{prop}[thm]{Proposition}

\theoremstyle{definition}
\newtheorem{defn}[thm]{Definition}
\newtheorem{rmk}[thm]{Remark}
\numberwithin{equation}{section}
\newcommand{\N}{\mathbb{N}}
\newcommand{\R}{\mathbb{R}}
\newcommand{\Q}{\mathbb{Q}}
\newcommand{\K}{\mathcal{K}}
\newcommand{\M}{\mathcal{M}}
\newcommand{\T}{\mathcal{T}}
\newcommand{\U}{\mathcal{U}}
\newcommand{\ep}{\varepsilon}

\newcommand{\C}{\mathfrak{C}}
\newcommand{\LC}{\mathfrak{LC}}
\newcommand{\SC}{\mathfrak{SC}}
\newcommand{\nbd}{\nobreakdash}
\newcommand{\nin}{{n\in\N}}
\newcommand{\nti}{{n\to\infty}}
\newcommand{\meas}{\mathrm{meas}}
\DeclareMathOperator{\supp}{supp}
\def\overcircle#1{\overset{\hbox{\tiny$\circ$}}{#1}}
\begin{document}
\title[]
{Topologies of continuity for\\  Carath\'eodory delay differential equations \\ with applications in non-autonomous dynamics}
\author[I.P.~Longo]{Iacopo P. Longo}
\address{Departamento de Matem\'{a}tica Aplicada, Universidad de
Valladolid, Paseo del Cauce 59, 47011 Valladolid, Spain.}
\email[Iacopo P. Longo]{iaclon@wmatem.eis.uva.es}
\email[Sylvia Novo]{sylnov@wmatem.eis.uva.es}
\email[Rafael Obaya]{rafoba@wmatem.eis.uva.es}
\thanks{Partly supported by MINECO/FEDER
under project MTM2015-66330-P and EU Marie-Sk\l odowska-Curie ITN Critical Transitions in
Complex Systems (H2020-MSCA-ITN-2014 643073 CRITICS)}
\author[S.~Novo]{Sylvia Novo}
\author[R.~Obaya]{Rafael Obaya}
\date{}
\begin{abstract}
We study some already introduced and some new strong  and weak topologies of integral type to provide continuous dependence on continuous initial data for the solutions of non-autonomous Carath\'eodory delay differential equations. As a consequence, we obtain new families of continuous  skew-product semiflows generated by delay differential equations whose vector fields  belong to such metric topological vector spaces of Lipschitz Carath\'eodory functions. Sufficient conditions for the equivalence of all or some of the considered strong or weak topologies are also given. Finally, we also provide results of continuous dependence of the solutions as well as of continuity of the skew-product semiflows generated by  Carath\'eodory delay differential equations when the considered phase space is a Sobolev space.
\end{abstract}
\maketitle
\section{Introduction}\label{secintro}
\noindent In this paper, we provide an extension of the theory on Carath\'eodory functions in order to allow the study of the solutions of  delay differential equations with finite delay of the type
\begin{equation}\label{eq:16.07-21:13}
\dot x=f\big(t,x(t),x(t-1)\big),
\end{equation}
where $f$ satisfies Carath\'eodory conditions. It is interesting to point out that, despite being a very particular class of delay differential problems, equations like \eqref{eq:16.07-21:13} are widely and successfully used in engineering and applied sciences to modelize many phenomena in which the past affects the future. Nevertheless, the lack of results of continuous variation of the solutions under Carath\'eodory conditions, prevented from applying many theoretical tools to analyze the qualitative behavior of the systems. Through the study of strong and weak metric topologies of integral type (part of which are genuinely new) and the analysis of the properties of some subsets of Carath\'eodory functions, we are able to prove several results of continuous dependence of the solutions with respect to initial data. Consequently, we provide the continuity of the skew-product semiflows composed of the flow on the hull of a Carath\'eodory function (satisfying appropriate assumptions) and of the solutions on a suitable phase space of the associated delay differential equation.
\par
\smallskip
The analogous query and some results for Carath\'eodory ordinary differential equations date back to the works by Artstein~\cite{paper:ZA1,paper:ZA2,paper:ZA3},  Heunis~\cite{paper:AJH}, Miller and Sell~\cite{book:RMGS, paper:RMGS1},  Neustadt~\cite{paper:LWN}, Opial~\cite{paper:O} and Sell \cite{paper:GS1, book:GS} among many others. However, despite its potential interest, such classic theory remained inconveniently incomplete. A large and systematic effort to make up for that, has been recently carried out by Longo et al. \cite{paper:LNO1, paper:LNO2}. Nevertheless, the same problems for Carath\'eodory delay differential equations remained still untreated. Relying on some of the structural and topological results in  \cite{paper:LNO1, paper:LNO2}, we fill such gap and pose new questions. It is worth noticing, however, that, as the problem becomes essentially infinite-dimensional, this is not a trivial extension of the previous results. Such framework opens  a wide range of dynamical scenarios in which it is possible to combine techniques of continuous skew-product flows, processes and random dynamical systems (see Arnold~\cite{book:LA}, Aulbach and  Wanner~\cite{paper:AW}, Caraballo and Han~\cite{book:CH},  Carvalho et al.~\cite{book:CLR}, Johnson et al.~\cite{book:JONNF}, P\"otzsche and Rasmussen~\cite{paper:CPMR}, Sacker and Sell \cite{sase}, Sell~\cite{book:GS},  Shen and Yi~\cite{paper:WSYY} and the references therein). As a consequence, it is possible to perform a richer qualitative analysis of  the local and global behavior of the solutions for such Carath\'eodory delay differential equations.  \par
\smallskip
Besides the introduction, the paper consists of three more sections which are organized as follows. In Section~\ref{sectopo}, we set  some preliminary notation and introduce the topological spaces which will be used throughout the rest of the work. Particularly, we recall the definitions of  the spaces of Lipschitz Carath\'eodory ($\LC$) and Strong Carath\'eodory ($\SC$) functions, and endow them with several possible topologies: besides the already established topologies (which are collectively analyzed in \cite{paper:LNO1, paper:LNO2}), such as  the classic $\T_B$, $\T_D$ (firstly presented in~\cite{book:RMGS, paper:RMGS1}) and $\sigma_D$ (also treated in ~\cite{paper:ZA1,paper:ZA2,paper:ZA3,paper:AJH,paper:LWN}), and the more recent  $\T_{\Theta}$ and $\sigma_{\Theta}$ (firstly presented in \cite{paper:LNO1} and \cite{paper:LNO2}, respectively), we introduce the new hybrid topologies $\T_{\Theta D}$, $\sigma_{\Theta D}$, $\T_{\Theta \smash{\widehat\Theta}}$, $\sigma_{\Theta \smash{\widehat\Theta}}$ and $\T_{\Theta B}$, where by hybrid we mean that they are, somehow, derived from the previous ones in a way that it is possible to  treat the first $N$ components of the spatial variable (representing the current state in a delay differential equation) in a different way from the last $N$ ones (representing the history of the state). The symbols $\Theta$ and $\widehat\Theta$  stand for sets of moduli of continuity which identify numerable quantities of compact sets of continuous functions on which the sequences of elements in $\LC$ are asked to converge. Furthermore, we show how such hybrid topologies relate to the first ones and how to apply or develop, in this new context, some of the  topological results obtained in  \cite{paper:LNO1, paper:LNO2} such as, the propagation of specific topological features on the $m$-bounds and the $l$-bounds, and the sufficient conditions for the equivalence of some or all of the introduced topologies.
\par
\smallskip
In Section~\ref{sec:Continuity-classic-topologies}, for each one of the previous topologies, we provide sufficient conditions on a subset $E$ of Lipschitz Carath\'eodory functions to have continuous dependence of the solutions with respect to initial data and to define a continuous skew-product semiflow on $E\times \mathcal{C}$, where $\mathcal{C}$ denotes the set of continuous functions mapping $[-1,0]$ onto $\R^N$. In particular, Subsection~\ref{subsec:classictopologies} deals with the result when the classic topologies $\T_B$, $\T_D$ and $\sigma_D$ are employed whereas Subsection~\ref{subsec:Ttx} treats the case in which the new hybrid topologies $\T_{\Theta B}$,  $\T_{\Theta D}$ and $\sigma_{\Theta D}$ are used. Notice also that we provide a theorem of continuity of the time-translations for all the new hybrid topologies.
\par
\smallskip
Section~\ref{sec:Sobolev} deals with the same problem, where the phase space $\mathcal{C}$ is now changed for the Sobolev space $\mathcal{C}^{1,p}$ of continuous functions from $[-1,0]$ to $\R^N$ which are differentiable almost everywhere and whose derivative is in $L^p([-1,0])$. In fact, such spaces are the natural environment to look at, because also in those cases in which the initial data is assumed to be just in $\mathcal{C}$, as soon as one looks at the solution for positive time, one has that $\dot x(t,f,\phi)=f\big(t,x(t),x(t-1)\big)\in L^p_{loc}$. Interestingly, if $p>1$, by finely tuning the choice of the sets of moduli of continuity $\Theta$ and $\widehat\Theta$, and under mild  assumptions on the  $m$-bounds of the functions, it is possible to retrieve a continuous skew-product semiflow for the topology $\T_{\Theta \smash{\widehat\Theta}}$. In order to obtain the same result for $\mathcal{C}^{1,1}$, assumptions on the $l$-bounds are necessary. It is worth noticing that reasoning as for these last results, we are also able to improve the information on the solutions obtained in the previous section if one disregards the interval $[-1,0]$, that is, the initial data. In other words, for any initial data in $\mathcal{C}$, we obtain the continuous variation in $\mathcal{C}^{1,p}([0,T])$ (resp. $\mathcal{C}^{1,1}([0,T])$) for any $T\ge0$ contained in the maximal interval of definition of the solution.
\section{Spaces and topologies}\label{sectopo}
\noindent In the following, we will denote by $\R^N$ the $N$\nbd-dimensional euclidean space with norm $|\cdot|$ and by $B_r$ the closed ball of $\R^N$ centered at the origin and with radius $r$. When $N=1$ we will simply write $\R$ and the symbol $\R^+$ will denote the set of positive real numbers. Furthermore, unless noted otherwise, $p$ is an integer in $[1,\infty)$, and for any interval $I\subseteq\R$ and any $W\subset\R^N$, we will use the following notation
\begin{itemize}[leftmargin=25pt]
\item[] $C(I,W)$: space of continuous functions mapping $I$ to $W$, endowed with the norm $\|\cdot\|_\infty$.
\item[] $C_C(\R)$: space of continuous functions with compact support in $\R$, endowed with the norm $\|\cdot\|_\infty$. When we want to restrict to the positive continuous functions with compact support in $\R$, we will write  $C^+_C(\R)$.
\item[] $L^p(I,\R^N)$, $1\le p \le \infty$: space of measurable functions from $I$ to $\R^N$ whose norm is in the Lebesgue space $L^p(I)$.
\item[] $L^p_{loc}(\R^N)$: the space of all functions $x(\cdot)$ of $\R$ into $\R^N$ such that for every compact interval $I\subset\R$, $x(\cdot)$ belongs to $L^p\big(I,\R^N\big)$. When $N=1$, we will simply write $L^p_{loc}$.
\end{itemize}
We will consider, and denote by $\C_p\big(\R^M,\R^N\big)$ (or simply $\C_p$ when $M=N$) the set of functions $f\colon\R\times\R^M\to \R^N$ satisfying
\begin{itemize}[leftmargin=25pt]
\item[(C1)] $f$ is Borel measurable and
\item[(C2)] for every compact set $K\subset\R^M$ there exists a real-valued function $m^K\in L^p_{loc}$, called \emph{$m$-bound} in the following, such that for almost every $t\in\R$ one has $|f(t,x)|\le m^K(t)$ for any $x\in K$.
\end{itemize}
As follows, we recall the definitions of the sets of Carath\'eodory functions which are subsequently used.
\begin{defn}\label{def:LC}
A function $f\colon\R\times\R^{M}\to \R^{N}$ is said to be \emph{Lipschitz Carath\'eodory}, and we will write $f\in \LC_p\big(\R^M,\R^N\big)$ (or simply $f\in\LC_p$ when $M=N$), if it satisfies (C1), (C2) and
\begin{itemize}[leftmargin=25pt]
\item[(L)] for every compact set $K\subset\R^M$ there exists a real-valued function $l^K\in L^p_{loc}$ such that for almost every $t\in\R$ one has $|f(t,x_1)-f(t,x_2)|\le l^K(t)|x_1-x_2|$ for any $x_1,x_2\in K$.
\end{itemize}
In particular, for any compact set $K\subset\R^M$, we refer to \emph{the optimal $m$-bound} and \emph{the optimal $l$-bound} of $f$ as to
\begin{equation}
m^K(t)=\sup_{x\in K}|f(t,x)|\qquad \mathrm{and}\qquad l^K(t)=\sup_{\substack{x_1,x_2\in K\\ x_1\neq x_2}}\frac{|f(t,x_1)-f(t,x_2)|}{|x_1-x_2|}\, ,
\label{eqOptimalMLbound}
\end{equation}
respectively. Clearly, for any compact set $K\subset\R^M$ the suprema in \eqref{eqOptimalMLbound}  can be taken for a countable dense subset of $K$ leading to the same actual definition, which guarantees that the functions defined in \eqref{eqOptimalMLbound} are measurable.
\end{defn}
\begin{defn}\label{def:SC}
A function $f\colon\R\times\R^M\to \R^N$ is said to be \emph{strong Carath\'eodory}, and we will write $f\in \SC_p\big(\R^M,\R^N\big)$ (or simply $f\in\SC_p$ when $M=N$), if it satisfies (C1), (C2) and
\begin{itemize}[leftmargin=25pt]
\item[(S)] for almost every $t\in\R$, the function $f(t,\cdot)$ is continuous.
\end{itemize}
The notion of \emph{optimal $m$-bound} for a strong  Carath\'eodory function on any compact set $K\subset\R^M$, is defined exactly as in equation \eqref{eqOptimalMLbound}.
\end{defn}
\begin{rmk}
As regards Definitions~\ref{def:LC} and ~\ref{def:SC},  when $p=1$, we will omit the number $1$ from the notation. For example, we will simply write $\LC$ instead of $\LC_1$. Moreover, the functions which lay in the same set and only differ on a negligible subset of  $\R^{1+N}$ will be identified; such a rule implies that $\LC_p\big(\R^M,\R^N\big)\subset\SC_p\big(\R^M,\R^N\big)$.
\end{rmk}
\begin{defn}
\label{def:l1l2bounds}
Let us consider a function $f\in\C_p(\R^{2N},\R^N)$. We say that \emph{$f$ admits $l_1$-bounds (resp. $l_2$-bounds)} if for every $j\in\N$ there exists a function $l_1^j(\cdot)\in L^p_{loc}$ (resp. $l_2^j(\cdot)\in L^p_{loc}$) such that for almost every $t\in\R$
\begin{equation*}
\begin{split}
|f(t,x_1,u)-f(t,x_2,u)|&\le  l_1^{j}(t)\,|x_1-x_2|\quad\text{for all }(x_1,u),(x_2,u)\in B_j\\[0.5ex]
\big(\text{resp. }|f(t,x,u_1)-f(t,x,u_2)|&\le  l_2^j(t)\,|u_1-u_2|\quad\text{for all }(x,u_1),(x,u_2)\in B_j\big).
\end{split}
\end{equation*}
In particular, if $f\in\LC_p(\R^{2N},\R^N)$, i.e. $f\colon\R\times\R^{2N}\to \R^N$ satisfying the assumptions in Definition~\ref{def:LC}, for every $j\in\N$ we refer to the \emph{optimal $l_1$-bound and the optimal $l_2$-bound for $f$ on $B_j\subset\R^{2N}$} as to
\begin{equation*}
\begin{split}
 l_1^{j}(t)&=\sup_{\substack{(x_1,u),(x_2,u)\in B_j\\ \\ x_1\neq x_2}}\frac{|f(t,x_1,u)-f(t,x_2,u)|}{|x_1-x_2|}\, ,\\
 l_2^j(t)&=\sup_{\substack{(x,u_1),(x,u_2)\in B_j\\ \\ u_1\neq u_2}}\frac{|f(t,x,u_1)-f(t,x,u_2)|}{|u_1-u_2|}\, .
\end{split}
\end{equation*}
If $f\in\SC_p(\R^{2N},\R^N)$ one can still define either the optimal $l_1$-bounds and/or the optimal $l_2$-bounds if, for almost every $t\in\R$, $ f$ is Lipschitz continuous with respect to the first and/or the last $N$ space variables, respectively.
\end{defn}
\begin{rmk}
Consider $f\in\LC_p(\R^{2N},\R^N)$. It is easy to prove that for all $t\in\R$ one has $l^j(t)\le l^j_1(t)+l^j_2(t)$, where by $l^j(\cdot)$ we denote the optimal $l$-bound for $f$ on $B_j$ as in \eqref{eqOptimalMLbound}.
\end{rmk}
We endow the space $\SC_p\big(\R^M,\R^N\big)$ with suitable strong and weak topologies. As a rule, when inducing a topology on a subspace we will denote the induced topology with the same symbol. Firstly we recall some integral-like topology which have been extensively used in the literature.
\begin{defn}[Topology $\T_{B}$]
We call $\T_{B}$ the topology on $\SC_p\big(\R^M,\R^N\big)$ generated by the family of seminorms
\begin{equation*}
p_{I,\, j}(f)=\sup_{x(\cdot)\in C(I,B_j)}\left[\int_I\big|f\big(t,x(t)\big)\big|^pdt \right]^{1/p},\quad f\in\SC_p\big(\R^M,\R^N\big)\, ,
\end{equation*}
where $I=[q_1,q_2]$, $q_1,q_2\in\Q$ and $j\in\N$. One has that
$\left(\SC_p\big(\R^M,\R^N\big),\T_{B}\right)$ is a locally convex metric space.
\end{defn}
\begin{defn}[Topologies $\T_D$ and $\sigma_D$]\label{def:TD}
Let $D$ be a countable and dense subset of $\R^M$. We call $\T_{D}$ (resp.  $\sigma_{D}$) the topology on $\SC_p\big(\R^M,\R^N\big)$  (resp. $\SC\big(\R^M,\R^N\big)$) generated by the family of seminorms
\begin{equation*}
p_{I,\, x}(f)=\left[\int_I|f(t,x_j)|^pdt \right]^{1/p}\qquad\left(\text{resp. }p_{I,\, x}(f)=\left|\,\int_If(t,x)\, dt\,\right|\right)
\end{equation*}
for $ f\in\SC_p\big(\R^M,\R^N\big)$ (resp. $ f\in\SC\big(\R^M,\R^N\big)$), $ x\in D,\, I=[q_1,q_2],\,  q_1,q_2\in\Q $. One has that
$\left(\SC_p\big(\R^M\big),\T_{D}\right)$ and $\left(\SC(\R^M,\R^N),\sigma_{D}\right)$ are locally convex metric spaces.
\end{defn}
Furthermore, we consider the two classes of metric topologies based on a suitable set of moduli of continuity which have been proposed in \cite{paper:LNO1} and \cite{paper:LNO2}, respectively. In order to do that, we recall the definition of suitable sets of moduli of continuity.
\begin{defn}[Suitable set of moduli of continuity]\label{def:ssmc}
We call  \emph{a suitable set of moduli of continuity}, any countable  set of non-decreasing continuous functions
\begin{equation*}
\Theta=\left\{\theta^I_j \in C(\R^+, \R^+)\mid j\in\N, \ I=[q_1,q_2], \ q_1,q_2\in\Q\right\}
\end{equation*}
such that $\theta^I_j(0)=0$ for every $\theta^I_j\in\Theta$, and  with the relation of partial order given~by
\begin{equation*}\label{def:modCont}
\theta^{I_1}_{j_1}\le\theta^{I_2}_{j_2}\quad \text{whenever } I_1\subseteq I_2 \text{ and } j_1\le j_2 \, .
\end{equation*}
\end{defn}
\begin{defn}[Topologies  $\T_{\Theta}$ and $\sigma_{\Theta}$]\label{def:TsigmaT}
Let $\Theta$ be a suitable set of moduli of continuity as in Definition~\ref{def:ssmc}, and $\K_j^I$ the compact set of functions in $C(I,B_j)$ which admit $\theta^I_j\in\Theta$ as a modulus of continuity. We call $\T_{\Theta}$ (resp. $\sigma_{\Theta}$)  the topology on $\SC_p\big(\R^M,\R^N\big)$  (resp. $\SC\big(\R^M,\R^N\big)$) generated by the family of seminorms
\begin{equation*}
p_{I,\, j}(f)=\!\!\sup_{x(\cdot)\in\K_j^I}\left[\int_I\big|f\big(t,x(t)\big)\big|^pdt \right]^{1/p} \ \left(\text{resp. }p_{I,\, j}(f)=\!\!\sup_{x(\cdot)\in\K_j^I}\left|\,\int_If\big(t,x(t)\big)\,dt\,\right|\right)
\end{equation*}
with $ f\in\SC_p\big(\R^M,\R^N\big)$ (resp. $ f\in\SC\big(\R^M,\R^N\big)$), $I=[q_1,q_2]$, $q_1,q_2\in\Q$, and $j\in\N$. One has that $\left(\SC_p\big(\R^M,\R^N\big),\T_{\Theta}\right)$ and $\left(\SC\big(\R^M,\R^N\big),\sigma_{\Theta}\right)$  are locally convex metric spaces.
\end{defn}
As follows, we introduce some new topologies on $\SC_p(\R^{2N},\R^N)$. We will  call them hybrid because they are derived from the ones presented above so that the condition of uniformity for the first $N$ space variables differ from the one of the remaining $N$ space variables.
\begin{defn}
\label{def:hybridtopologies}
Let $\Theta$ and $\smash{\widehat\Theta}$ be suitable sets of moduli of continuity as defined in Definition~\ref{def:ssmc}, $D$ be a countable dense subset of $\R^N$ and, for any $I=[q_1,q_2]$, $q_1,q_2\in\Q$ and $j\in\N$, let $\K_j^I$ and $\widehat\K_j^I$ be the sets of functions in $C(I,B_j)$ which admit $\theta^I_j$ and $\hat\theta_j^I$, respectively, as a moduli of continuity.
\begin{itemize}[leftmargin=10pt]
\item We call $\T_{\Theta D}$ (resp. $\sigma_{\Theta D}$) the topology on $\SC_p(\R^{2N},\R^N)$ (resp. $\SC(\R^{2N},\R^N)$) generated by the family of seminorms
\begin{equation*}
\begin{split}
p_{I,\, u,\, j}(f)=\sup_{x(\cdot)\in\K_j^I}\left[\int_I\big|f\big(t,x(t),u\big)\big|^pdt \right]^{1/p}  ,\\
 \Bigg(\text{resp. }p_{I,\,u,\, j}(f)=\sup_{x(\cdot)\in\K_j^I}\left|\int_If\big(t,x(t),u\big)\,dt\right|\Bigg)
\end{split}
\end{equation*}
with $f\in\SC_p(\R^{2N},\R^N)$ (resp. $f\in\SC(\R^{2N},\R^N)$), $I=[q_1,q_2]$, $q_1,q_2\in\Q$, $u\in D$ and $j\in\N$. One has that
$\left(\SC_p(\R^{2N},\R^N),\T_{\Theta D}\right)$ and $\left(\SC(\R^{2N},\R^N),\sigma_{\Theta D}\right)$ are locally convex metric spaces.\vspace{0.2cm}
\item We call $\T_{\Theta \smash{\widehat\Theta}}$ (resp. $\sigma_{\Theta \smash{\widehat\Theta}}$) the topology on $\SC_p(\R^{2N},\R^N)$ (resp. $\SC(\R^{2N},\R^N)$)  generated by the family of seminorms
\begin{equation*}
\begin{split}
 p_{I,\, j}(f)=\sup_{x(\cdot)\in\K_j^I,\,u(\cdot)\in\widehat\K_j^{I-1}}\left[\int_I\big|f\big(t,x(t),u(t-1)\big)\big|^pdt \right]^{1/p}  ,\\
\Bigg(\text{resp. }  p_{I,\, j}(f)=\sup_{x(\cdot)\in\K_j^I,\,u(\cdot)\in\widehat\K_j^{I-1}}\left|\int_If\big(t,x(t),u(t-1)\big)\,dt\right| \Bigg)
\end{split}
\end{equation*}
with  $f\in\SC_p(\R^{2N},\R^N)$ (resp. $f\in\SC(\R^{2N},\R^N)$), $I=[q_1,q_2]$, $q_1,q_2\in\Q$ and $j\in\N$. One has that
$\left(\SC_p(\R^{2N},\R^N),\T_{\Theta \smash{\widehat\Theta}}\right)$ and $\left(\SC(\R^{2N},\R^N),\sigma_{\Theta \smash{\widehat\Theta}}\right)$ are locally convex metric spaces.\vspace{0.2cm}
\item We call $\T_{\Theta B}$ the topology on $\SC_p(\R^{2N},\R^N)$ generated by the family of seminorms
\begin{equation*}
  p_{I,\, j}(f)=\sup_{x(\cdot)\in\K_j^I,\,u(\cdot)\in C(I-1, B_j)}\left[\int_I\big|f\big(t,x(t),u(t-1)\big)\big|^pdt \right]^{1/p}  \, ,
\end{equation*}
with  $f\in\SC_p(\R^{2N},\R^N)$,  $I=[q_1,q_2]$, $q_1,q_2\in\Q$ and $j\in\N$. One has that
$\left(\SC_p(\R^{2N},\R^N),\T_{\Theta B}\right)$ is a locally convex metric space.
\end{itemize}
\end{defn}
As one may notice, the way the topology $\sigma_{\Theta\smash{\widehat\Theta}}$ is defined in Definition~\ref{def:hybridtopologies}, requires that the interval on which we take the integral is the same as, or it is its translation by $1$, of the domain of the functions over which we take the supremum. However, in many cases we will need to consider the integral on a smaller subinterval that we can not directly control with the integral on the whole interval due to the employed weak formulation. Nevertheless, playing with the functions in the compact sets $\K_j^J$ and $\widehat\K_j^J$, we are still able to achieve the properties we need, as shown in the technical lemma below. Such result is the analogous of Lemma 2.13(ii) in \cite{paper:LNO2} for $\sigma_\Theta$, and since the two proofs differ only on minor details, we skip the proof.
\begin{lem}\label{lem:conv-subintHYBRID}
Let $\Theta$ and $\widehat \Theta$ be suitable sets of moduli of continuity as in {\rm Definition~\ref{def:ssmc}} and, for each $j\in\N$ and  $I=[q_1,q_2]$, $q_1,q_2\in\Q$, let $\K_j^I$ and $\widehat\K_j^I$ be the compact sets in $C(I,B_j)$  which admit $\theta^I_j\in\Theta$ and $\hat\theta_j^I\in\widehat\Theta$, respectively, as a modulus of continuity.
  If $(g_n)_\nin$ is a sequence in $\SC(\R^{2N},\R^N)$ converging to some function $g$ in $\left(\SC(\R^{2N},\R^N),\sigma_{\Theta\smash{\widehat\Theta}}\right)$, then one has that
\[ \lim_{n\to\infty}\sup_{x(\cdot)\in\K_j^I,\ u(\cdot)\in\widehat\K_j^{I-1}}\left| \int_{p_1}^{p_2} \!\!\big[g_n\big(t,x(t),u(t-1)\big)- g\big(t,x(t),u(t-1)\big)\big]\,dt\right|=0,\]
whenever $p_1$, $p_2\in \Q$ and $q_1\le p_1 < p_2\le q_2$.
\end{lem}
The previous lemma allows to obtain a characterization of the topologies $\T_{\Theta \smash{\widehat\Theta}}$, $\sigma_{\Theta \smash{\widehat\Theta}}$ and $\T_{\Theta B}$ presented in Definition \ref{def:hybridtopologies}. The reason for proving that these same topologies can also be induced  by alternative families of seminorms is that, even though the specific way in which we defined them is particularly useful when dealing with solutions of delay differential equations with finite delay of the type~\eqref{eq:16.07-21:13}, some other topological results in the following can be simplified considerably if these alternative forms are used.
\begin{lem}\label{lem:alternative-seminorms}
Let $\Theta$ and $\smash{\widehat\Theta}$ be suitable sets of moduli of continuity as in {\rm Definition~\ref{def:ssmc}}, and $D$ be a countable dense subset of $\R^N$. With the notation used in {\rm Definition \ref{def:hybridtopologies}} the following statements hold.
\begin{itemize}[leftmargin=10pt]
\item The topology  $\sigma_{\Theta \smash{\widehat\Theta}}$ on $\SC(\R^{2N},\R^N)$ is also generated by the family of seminorms
\begin{equation}\label{eq:sigma-alternative-seminorms}
p_{I,\, j}(f)=\sup_{x(\cdot)\in\K_j^I,\,u(\cdot)\in\widehat\K_j^{I}}\left|\int_I f\big(t,x(t),u(t)\big)\,dt\right|
\end{equation}
with  $f\in\SC(\R^{2N},\R^N)$,  $I=[q_1,q_2]$, $q_1,q_2\in\Q$ and $j\in\N$. \vspace{0.2cm}
\item The topologies $\T_{\Theta \smash{\widehat\Theta}}$ and $\T_{\Theta B}$ on $\SC_p(\R^{2N},\R^N)$ are also generated by the family of seminorms
\begin{equation*}
\begin{split}
 &p_{I,\, j}(f)=\sup_{x(\cdot)\in\K_j^I,\,u(\cdot)\in\widehat\K_j^{I}}\left[\int_I\big|f\big(t,x(t),u(t)\big)\big|^pdt \right]^{1/p}  ,\\
\text{and  }   \qquad  &q_{I,\, j}(f)=\sup_{x(\cdot)\in\K_j^I,\,u(\cdot)\in C(I, B_j)}\left[\int_I\big|f\big(t,x(t),u(t)\big)\big|^pdt \right]^{1/p}  \, ,
\end{split}
\end{equation*}
respectively, with  $f\in\SC_p(\R^{2N},\R^N)$, $I=[q_1,q_2]$, $q_1,q_2\in\Q$ and $j\in\N$.
\end{itemize}
\end{lem}
\begin{proof}
Let $\Theta$ and $\smash{\widehat\Theta}$ be suitable sets of moduli of continuity and consider a sequence $(f_n)_\nin$  in $\SC(\R^{2N},\R^N)$ converging to some $f\in\SC(\R^{2N},\R^N)$ with respect to the topology $\sigma_{\Theta\smash{\widehat\Theta}}$. We shall prove that $(f_n)_\nin$ converges to $f$ with respect to the topology generated by the family of seminorms in \eqref{eq:sigma-alternative-seminorms}. Let us fix $j\in\N$ and any interval $I$ with rational extrema, and consider another interval $J$ with rational extrema such that  $I\cup(I-1)\subset J-1$. Let us extend  the functions $x(\cdot)\in \K^I_j$ by constants to $J$, and up to changing any function $u(\cdot)\in \widehat\K^I_j$ for its translation $\tilde u:I-1\to B_j\subset \R^N$ so that $\tilde u(t)=u(t+1)$ for all $t\in I-1$, let us also extend  such  functions $\widetilde u$ by constant to $J-1$. Then, one has that
\begin{align}\label{eq:09.11-11:42}
\begin{split}
&\sup_{x(\cdot)\in\K^I_{j},\  u(\cdot)\in\widehat\K^I_{j}}\bigg|\int_I[f_n\big(t,x(t),u(t)\big)- f\big(t,x(t),u(t)\big)\big]\,dt\bigg|\\
&\qquad\le\sup_{x(\cdot)\in\K^J_{j},\ u(\cdot)\in\widehat\K^{J-1}_{j}}\bigg|\int_I[f_n\big(t,x(t),u(t-1)\big)- f\big(t,x(t),u(t-1)\big)\big]\,dt\bigg|\,.\\
\end{split}
\end{align}
Now, by assumption we have that
\begin{equation*}
\sup_{x(\cdot)\in\K^J_j,\ u(\cdot)\in\widehat\K^{J-1}_j}\bigg|\int_J[f_n\big(t,x(t),u(t-1)\big)- f\big(t,x(t),u(t-1)\big)\big]\,dt\bigg|\xrightarrow{\nti}0,
\end{equation*}
and thus, passing to the limit as $\nti$ in \eqref{eq:09.11-11:42}, we obtain the result thanks to Lemma~\ref{lem:conv-subintHYBRID} and recalling that $I\subset J$.
\par\smallskip
On the other hand, if $(f_n)_\nin$ is a sequence in $\SC(\R^{2N},\R^N)$ converging to some $f\in\SC(\R^{2N},\R^N)$ with respect to the topology generated by the family of seminorms in \eqref{eq:sigma-alternative-seminorms}, then fixed $j\in\N$ and any interval $I$ with rational extrema, one can consider an interval $J$ with rational extrema such that $I\cup(I-1)\subset J$ and thus, up to extending the functions in $\K^I_{j}$ and $\widehat\K^{I-1}_{j}$ by constants to $J$ one has
\begin{align*}
&\sup_{x(\cdot)\in\K^I_{j},\ u(\cdot)\in\widehat\K^{I-1}_{j}}\bigg|\int_I[f_n\big(t,x(t),u(t-1)\big)- f\big(t,x(t),u(t-1)\big)\big]\,dt\bigg|\\
&\qquad\qquad \le\sup_{x(\cdot)\in\K^J_{j},\  u(\cdot)\in\widehat\K^J_{j}}\bigg|\int_I[f_n\big(t,x(t),u(t)\big)- f\big(t,x(t),u(t)\big)\big]\,dt\bigg|
\end{align*}
Therefore, since a result analogous to  Lemma~\ref{lem:conv-subintHYBRID} holds also for the convergence with respect to the topology generated by the family of seminorms in \eqref{eq:sigma-alternative-seminorms}, one obtains that, by taking the limit as $\nti$, $(f_n)_\nin$ converges to  $f$ with respect to the  topology $\sigma_{\Theta\widehat\Theta}$.
\par
\smallskip
As regards the topologies $\T_{\Theta \smash{\widehat\Theta}}$ and $\T_{\Theta B}$, one can simplify the previous reasoning, due to the fact that the integral on any interval can be directly controlled by above with an integral on a larger interval.
\end{proof}
A natural question arises, concerning the relation between the topologies $\T_{\Theta}$ and $\T_{\Theta\Theta}$ (resp. $\sigma_{\Theta}$ and $\sigma_{\Theta\Theta}$) on $\SC_p(\R^{2N},\R^N)$ (resp. $\SC(\R^{2N},\R^N)$). The following proposition uses the previous lemma to obtain a relation of order.
\begin{prop}\label{prop:Theta=ThetaTheta}
Let $\Theta$ be a suitable set of moduli of continuity as in {\rm Definition~\ref{def:ssmc}}. Considered the topologies $\T_\Theta$ and  $\T_{\Theta\Theta}$ on $\SC_p(\R^{2N},\R^N)$, and, $\sigma_\Theta$ and $\sigma_{\Theta\Theta}$ on $\SC(\R^{2N},\R^N)$  (see {\rm Definitions~\ref{def:TsigmaT}} and {\rm~\ref{def:hybridtopologies}})  the following order relations hold:
\begin{equation*}
\T_{\Theta}\le  \T_{\Theta\Theta}\qquad\text{and}\qquad \sigma_{\Theta}\le \sigma_{\Theta\Theta}.
\end{equation*}
\end{prop}
\begin{proof}
We will complete the proof for the weak topologies because the other one is analogous (and simpler). Furthermore, in order to avoid any abuse of notation, within this proof we will write
\begin{align*}
\K^I_{j,2N}&:=\{\xi\colon I\to B_j\subset\R^{2N}\mid |\xi(t)-\xi(s)|\le\theta^I_j (|t-s|),\text{ for all }t,s\in I\},\\[1ex]
\K^I_{j,N}&:=\{\eta\colon I\to B_j\subset\R^{N}\mid |\eta(t)-\eta(s)|\le\theta^I_j (|t-s|),\text{ for all }t,s\in I\}.
\end{align*}
Let $(f_n)_\nin$ be a sequence  in $\SC(\R^{2N},\R^N)$ converging to some $f\in\SC(\R^{2N},\R^N)$ with respect to the topology $\sigma_{\Theta\Theta}$ and prove that one also has $f_n\xrightarrow{\sigma_{\Theta}}f$ as $\nti$. To the aim, fix $I=[q_1,q_2]$, $q_1,q_2\in\Q$, and $j\in\N$. Then, one has
\begin{align*}\label{eq:20.06-20:39}
\begin{split}
&\!\!\!\!\!\sup_{(x(\cdot),u(\cdot))\in\K^I_{j,2N}}\bigg|\int_I\big[f_n\big(t,x(t),u(t)\big)- f\big(t,x(t),u(t)\big)\big]\,dt\bigg|\\
&\quad\ \ \ \le\!\sup_{x(\cdot)\in\K^I_{j,N},\  u(\cdot)\in\K^I_{j,N}}\bigg|\int_I[f_n\big(t,x(t),u(t)\big)- f\big(t,x(t),u(t)\big)\big]\,dt\bigg|\,.
\end{split}
\end{align*}
Now, by assumption and thanks to Lemma~\ref{lem:alternative-seminorms} we have that the right-hand side of the previous inequality goes to zero as $\nti$, which ends the proof.
\end{proof}
\begin{rmk}\label{rmk:DelayTopologiesChain}
Consider any dense and countable set $D\subset\R^M$, and any pair $\Theta$ and $\smash{\widehat\Theta}$ of suitable sets of moduli of continuity as in Definition~\ref{def:ssmc}, such that for any $I=[q_1,q_2]$, $q_1,q_2\in\Q$ and $j\in\N$ one has
\begin{equation*}
\theta^I_j(t)\le\hat\theta^I_j(t),\qquad\text{for all }t\in[0,\infty).
\end{equation*}
Then, one can draw the following chains of order:
\begin{align}\label{eq:DelayTopologiesChain}
\begin{split}
\sigma_D\le\T_{D}&\le\T_{\Theta D}\le\T_{\Theta}\le\T_{\Theta\Theta}\le\T_{\Theta \smash{\widehat\Theta}}\le\T_{\Theta B}\le\T_{B}\quad\text{and}\\[0.5ex]
&\sigma_D\le\sigma_{\Theta D}\le \sigma_{\Theta}\le\sigma_{\Theta\Theta}\le \sigma_{\Theta \smash{\widehat\Theta}}\le\T_{\Theta \smash{\widehat\Theta}},
\end{split}
\end{align}
where, in particular, the order relations $\, \T_{\Theta}\le\T_{\Theta\Theta}\, $ and $\, \sigma_{\Theta}\le\sigma_{\Theta\Theta}\, $ hold true thanks to Proposition~\ref{prop:Theta=ThetaTheta}.  Clearly, one might expand the previous chains of inequalities (or generate new branches) by considering more suitable sets of moduli of continuity (satisfying appropriate relations of partial order) and/or different dense countable subsets of $\R^N$, and the corresponding induced topologies.
\end{rmk}
The choice of the topological space will deeply affect the construction of the relative skew-product semiflow. In particular, let us recall the definition of hull  for a function.
\begin{defn}[Hull of a function]\label{def:Hull}
Let $(E,d)$ be a metric space of functions mapping $\R\times X$ onto $X$, where $X$ is a metric space, and let $\T$ be the topology induced by the metric. If $f\in E$, and for any $t\in\R$ also $f_t\in E$, where $f_t$ is the time translation at time $t$ of $f$, i.e. the function
\begin{equation}\label{eq:time-translations}
f_t\colon\R\times X\to X,\qquad(s,x)\mapsto f_t(s,x)=f(s+t,x),
\end{equation}
then we call  \emph{the hull of $f$ with respect to $(E,\T)$}, the metric subspace of  $(E,\T)$ defined~by
\begin{equation*}
\mathrm{Hull}_{(E,\T)}(f)=\big(\mathrm{cls}_{(E,\T)}\{f_t\mid t\in\R\} ,\, \T\big)\, ,
\end{equation*}
where, $\mathrm{cls}_{(E,\T)}(A)$ represents the closure in $(E,\T)$ of the set $A$ and $\T$ is the induced topology.
\end{defn}
In the last part of the section, we recall the notions of $L^1_{loc}$-equicontinuity and $L^p_{loc}$-boundedness, relate them to Carath\'eodory functions and prove some results on the previously outlined topological spaces once such properties are assumed to hold.
A subset $S$ of positive functions in $L^p_{loc}$  is bounded if for every $r>0$ the following inequality holds
\begin{equation*}
\sup_{m\in S}\int_{-r}^r \big(m(t)\big)^p\,  dt<\infty\,.
\end{equation*}
In such a case we will say that $S$ is $L^p_{loc}$-bounded.
\begin{defn}\label{weakcomp}
A set $S$ of positive functions in $L^1_{loc}$ \emph{is $L^1_{loc}$-equicontinuous} if for any $r>0$ and for any $\ep>0$ there exists a $\delta=\delta(r,\ep)>0$ such that, for any $-r\le s\le t\le r$, with $t-s<\delta$, the following inequality holds
\begin{equation*}
\sup_{m\in S}\int_{s}^t m(u)\,du<\ep\, .
\end{equation*}
\end{defn}
\begin{rmk}\label{rmk:equicnt=>bound}
Notice that the $L^1_{loc}$-equicontinuity implies the $L^1_{loc}$-boundedness. On the other hand, due to H\"older inequality, if $p>1$ the $L^p_{loc}$-boundedness implies  the $L^1_{loc}$-equicontinuity.
\end{rmk}
The following definition extends the notions of $L^1_{loc}$-equicontinuity and $L^p_{loc}$-boundedness to sets of Carath\'eodory functions through their $m$-bounds, $l$-bounds, and/or $l_2$-bounds. Recall that, by time translation at time $t$ of a function $f$, we mean the application $f_t\colon\R\times\R^M\to\R^N$ defined  in \eqref{eq:time-translations}.
\begin{defn}
We say that
\begin{itemize}[leftmargin=20pt]\setlength\itemsep{0.4em}
\item[(i)] a set $E\subset\SC_p(\R^M,\R^N)$ \emph{admits $L^p_{loc}$-bounded (resp. $L^1_{loc}$-equicontinuous)} $m$\nbd-bounds, if for any $j\in\N$ the set $S^j\subset L^p_{loc}$, made of the optimal $m$-bounds on $B_j$  of the functions in $E$, is $L^p_{loc}$\nbd-bounded (resp. $L^1_{loc}$-equicontinuous);
\item[(ii)] \emph{$f\in\SC_p(\R^M,\R^N)$ has $L^p_{loc}$-bounded (resp. $L^1_{loc}$-equicontinuous) $m$-bounds} if the set $\{f_t\mid t\in\R\}$ admits  $L^p_{loc}$-bounded (resp. $L^1_{loc}$-equicontinuous) $m$-bounds;
\item[(iii)]a set $E\subset\LC_p(\R^M,\R^N)$  \emph{admits  $L^p_{loc}$-bounded  (resp. $L^1_{loc}$-equicontinuous)} $l$\nbd-bounds, if for any $j\in\N$, the set $S^j\subset L^p_{loc}$, made of the optimal $l$-bounds on $B_j$  of the functions in $E$,  is  $L^p_{loc}$-bounded  (resp. $L^1_{loc}$-equicontinuous);
\item[(iv)] $f\in\LC_p(\R^M,\R^N)$  \emph{has $L^p_{loc}$-bounded  (resp. $L^1_{loc}$-equicontinuous) $l$-bounds} if  the set $\{f_t\mid t\in\R\}$  has $L^p_{loc}$-bounded (resp. $L^1_{loc}$-equicontinuous) $l$-bounds;
\item[(v)]a set $E\subset\SC_p(\R^M,\R^N)$  \emph{admits  $L^p_{loc}$-bounded  (resp. $L^1_{loc}$-equicontinuous) $l_i$\nbd-bounds}, with $i\in\{1,2\}$, if for any $j\in\N$, the set $S^j\subset L^p_{loc}$, made  of the optimal $l_i$-bounds on $B_j$  of the functions in $E$,  is  $L^p_{loc}$-bounded  (resp. $L^1_{loc}$-equicontinuous);
\item[(vi)] $f\in\SC_p(\R^M,\R^N)$  \emph{has $L^p_{loc}$-bounded  (resp. $L^1_{loc}$-equicontinuous) $l_i$-bounds}, with $i\in\{1,2\}$, if  the set $\{f_t\mid t\in\R\}$  has $L^p_{loc}$-bounded (resp. $L^1_{loc}$-equiconti\-nuous) $l_i$-bounds.
\end{itemize}
\label{def:05.07-13:05}
\end{defn}
As follows, we prove that the properties defined above are propagated through the limits in the considered topologies. The proof of this proposition employs  arguments used in some results of Section 4 of \cite{paper:LNO1}, yet we include a full proof for clarity and because some formulas will be useful in the following.
\begin{prop}\label{prop:07.07-19:44p=1}
Let $E$ be a subset of $ \SC_p\big(\R^{M},\R^N\big)$ with $L^p_{loc}$-bounded (resp. $L^1_{loc}$-equicontinuous) $m$-bounds.  Then $\mathrm{cls}_{(\SC_p(\R^{M},\R^N),\T)}(E)$ admits $L^p_{loc}$-bounded  (resp. $L^1_{loc}$-equicontinuous) $m$-bounds, where $\T$ is either, any of the topologies in \eqref{eq:DelayTopologiesChain} if $p=1$, or any of the strong topologies in \eqref{eq:DelayTopologiesChain} if $p>1$. The same statement is true for the $l$-bounds and, if $M=2N$, also for the $l_1$-bounds  and the $l_2$-bounds.
\end{prop}
\begin{proof}
Let us firstly consider $p=1$ and the case of  $L^1_{loc}$-bounded $m$-bounds. By \eqref{eq:DelayTopologiesChain}, if we prove the result for the topology $\sigma_D$, where $D$ is any countable dense subset of $\R^M$, we have it for all the other topologies. Let us firstly reason for the $m$-bounds and for $p=1$. Moreover, in order to simplify the notation, let $\overline E$ denote the set $\mathrm{cls}_{(\SC(\R^M,\R^N),\sigma_D)}(E)$. As we are applying the topological closure in $(\SC(\R^M,\R^N),\sigma_D)$, we already know that each function  in $\overline E$ admits an optimal $m$-bound. However, without any additional piece of  information, it is difficult to say whether the optimal $m$-bounds of the limit functions allow us to preserve the property of $L^1_{loc}$-boundedness in $\overline E$ or not.
\par
\smallskip
Fix $j\in\N$ and, for any $\nin$, let $m^j_n$ be the optimal $m$-bound for $f_n$ on $B_j$ and  $\mu^j_n$ be the positive absolutely continuous measure (with respect to the Lebesgue measure) with density $m^j_n(\cdot)$. By hypothesis, the set $\{ m^j_n(\cdot)\mid\nin\}$ is $L^1_{loc}$-bounded. Hence, due to Theorems 15.7.5 in Kallenberg~\cite{book:Kall}, the sequence of induced measures $(\mu^j_n)_\nin$,  is relatively compact in $(\M^+,\widetilde{\sigma})$ (set of positive and regular Borel measures on $\R$ endowed with the vague topology; see  \cite {book:Kall} for more information), and thus it vaguely converges, up to a subsequence, to a measure $\mu^j\in\M^+$, i.e.
\begin{equation*}
 \lim_{\nti}\int_\R \phi(s)\, d\mu_n(s)=\int_\R\phi(s)\, d\mu(s)\qquad \text{for all } \phi\in C^+_C(\R).
\end{equation*}
Moreover, by Lebesgue-Besicovitch differentiation theorem, there exists $m^j(\cdot)\in L^1_{loc}$ such that
\begin{equation}
m^j(t)=\lim_{h\to0}\frac{\mu^j([t,t+h])}{h}\, , \qquad \mathrm{for\ a.e.} \ t\in\R\, ,
\label{eq:19.05-12:37}
\end{equation}
and $m^j(\cdot)$ is the density of the absolutely continuous part of the Radon-Nikod\'ym decomposition of $\mu^j$ in each compact interval. We claim that $m^j(\cdot)$ is an $m$-bound for $f$ on $B_j$. Let us firstly fix $x\in D\cap B_j$, and take $t,h\in\Q$, with $h>0$, and $\phi\in C_C^+(\R)$ such that $\phi\equiv 1$ in $[t,t+h]$. Then, we have
\begin{equation*}
\begin{split}
\bigg|\frac{1}{h}\int_{t}^{t+h}f(s,x)\, ds\,\bigg|&=\lim_{\nti}\bigg|\frac{1}{h}\int_{t}^{t+h}f_n(s,x)\, ds\,\bigg|\le\lim_{\nti}\frac{1}{h}\int_{t}^{t+h}m^j_n(s)\, ds \\
&\le \lim_{\nti}\frac{1}{h}\int_\R\phi(s)\, d\mu^j_n(s)=\frac{1}{h}\int_\R\phi(s)\, d\mu^j(s)\,.
\end{split}
\end{equation*}
Moreover, thanks to the regularity of $\mu^j$, one has
\begin{equation*}
 \mu^j([t,t+h])=\inf\left\{\int_\R\phi(s)\, d\mu^j(s)\;\Big|\; \phi\in C_C^+(\R),\;  \phi\equiv 1\; \text{in } [t,t+h]\right\}.
\end{equation*}
Therefore, from the previous two formulas, we obtain
\begin{equation}\label{eq:27-06_12:23}
\bigg|\frac{1}{h}\int_{t}^{t+h}f(s,x)\, ds\,\bigg|\le\frac{\mu^j([t,t+h])}{h}\,.
\end{equation}
Now, consider $t,h\in\R$, with $h>0$, and let $(s_n)_\nin$ and $(t_n)_\nin$ be two sequences in $\Q$ such that, as $\nti$, $s_n\downarrow t$ and $t_n\uparrow t+h$, respectively. By \eqref{eq:27-06_12:23}, applied on the intervals $[s_n,t_n]$, and noticing that $\mu^j([s_n,t_n)]\le \mu^j([t,t+h])$ for every $\nin$, one can write
\begin{equation*}
\left|\,\frac{1}{h}\int_{s_n}^{t_n}f(s,x)\, ds\,\right|\le\frac{\mu^j([t,t+h])}{h}\,,\quad\text{for all }\nin\,.
\label{eq:13-07_14:05}
\end{equation*}
Hence, passing to the limit as $\nti$ and using the continuity of the integral, one obtains \eqref{eq:27-06_12:23} for every $t,h\in\R$ with $h>0$. Now, as $h\to0$ (see Dunford and Schwartz \cite[Corollary III.12.7, p.216]{book:DS}) and using \eqref{eq:19.05-12:37}, we obtain that for almost every $t\in\R$,
\begin{equation}
|f(t,x)|\le m^j(t)\,,
\label{eq:21.05-12:41}
\end{equation}
for the fixed $x\in D\cap B_j$. For every fixed $x\in D\cap B_j$ let us now denote by $R(x)$ the subset of $\R$ such that $\meas_\R(\R\setminus R(x))=0$ and \eqref{eq:21.05-12:41}  holds for all $t\in R(x)$. Such a set clearly depends on $x\in D\cap B_j$. However, since $D$ is numerable, by simply intersecting all the possible $R(x)$, with $x\in D\cap B_j$,  one can obtain a set $R_0\subset\R$ of full measure for which \eqref{eq:21.05-12:41}  holds for any $x\in D\cap B_j$. Finally, by the continuity of $f(t,\cdot)$, we obtain the result for almost every $t\in\R$ for all $x\in B_j$, and $m^j$ provides an $m$-bound for $f$ in $B_j$, as claimed.\par\smallskip
Now, we prove that $\overline E$ admits $L^1_{loc}$-bounded $m$-bounds. For each $f\in \overline E$ for each $f\in \overline E$, let $(f_n)_\nin$ be a sequence in $E$ converging to $f$ with respect to $\sigma_D$ and for any $j\in\N$, let $m_f^j$ be either, the optimal $m$-bound of $f$ on $B_j$ if $f\in E$, or the $m$-bound given by \eqref{eq:19.05-12:37}, i.e. the absolutely continuous part of a limit measure through sequence of the $m$-bounds of $(f_n)_\nin$, if $f\in \overline E\setminus E$.
\par\smallskip
Consider $j\in\N$, $r,\delta>0$ and $\phi\in C^+_C$ such that $\supp \phi\subset [-r-\delta,r+\delta]$ and $\phi\equiv 1$ in $[-r,r]$, then, we have
\begin{equation}\label{eq:boundedMbounds}
\begin{split}
\int_{-r}^r m_f^j(t)\, dt&\le\int_\R\phi(t)\,  m_f^j(t)\, dt\le\lim_\nti \int_\R \phi(t)\, m_{f_n}^{j}(t)\, dt\\
&\le\sup_{g\in E}\int_{-r-\delta}^{r+\delta} \, m_{g}^{j}(t)\, dt\, ,
\end{split}
\end{equation}
where the last inequality comes from the assumption of  $L^1_{loc}$-boundedness for the $m$-bounds of $E$. The chain of inequalities in \eqref{eq:boundedMbounds} already allows to prove that $\overline E$ admits $L^1_{loc}$-bounded $m$-bounds, but a qualitative refinement is actually possible. By the continuity of the Lebesgue integral, for any $f\in \overline E$, $r>0$  and $\ep>0$ there exists $\delta=\delta(f,r,\ep)>0$ such that
\begin{equation*}
\int_{-r}^r m_f^j(t)\, dt\le\ep+\int_{-r+\delta}^{r-\delta} m_f^j(t)\, dt\le \ep+\sup_{g\in E}\int_{-r}^{r} \, m_{g}^{j}(t)\, dt,
\end{equation*}
where the last inequality is achieved using \eqref{eq:boundedMbounds}. Then, by the arbitrariness of $\ep$ first  and  by  taking the superior on $f\in\overline E$ on both sides afterwards, one actually obtains that, not only $\overline E$ has $L^1_{loc}$-bounded $m$-bounds, but also that, fixed any interval $[-r,r]$, the constant provided by the $L^1_{loc}$-boundedness of the $m$-bounds of $E$ also applies to the $m$-bounds of $\overline E$, i.e.
\begin{equation}\label{eq:boundedMbounds-sameCONSTANT}
\int_{-r}^r m_f^j(t)\, dt\le \sup_{g\in E}\int_{-r}^{r} \, m_{g}^{j}(t)\, dt.
\end{equation}
\par\smallskip
Let us now assume that $E$ has $L^1_{loc}$-equicontinuous $m$-bounds. As before, for each $f\in\overline E$, let $(f_n)_\nin$ be a sequence in $E$ converging to $f$ with respect to $\sigma_D$ and for any $j\in\N$ let $m_f^j$ be either, the optimal $m$-bound of $f$ on $B_j$ if $f\in E$, or the $m$-bound given by \eqref{eq:19.05-12:37} if $f\in\overline E\setminus E$, i.e. the absolutely continuous part of a limit measure. By the $L^1_{loc}$-equicontinuity of the $m$-bounds, we have that for each $j\in\N$, and $r,\ep>0$ there exists $\delta=\delta(r,\ep)>0$ such that
\begin{equation*}
\text{for all } t,s\in[-r,r]\, :\quad 0<t-s<\delta\quad\Rightarrow\quad \sup_{g\in E}\int_s^tm^j_g(u)\,du<\ep.
\end{equation*}
Thus, considered $t,s\in[-r,r]$ with $s<t$ and $t-s<\delta$ and reasoning as for \eqref{eq:boundedMbounds} and \eqref {eq:boundedMbounds-sameCONSTANT} one obtains the aimed inequality which implies that  $\overline E$ admits $L^1_{loc}$-equicontinuous $m$-bounds. \par\smallskip
Let us now treat the case $p>1$.  By \eqref{eq:DelayTopologiesChain}, if we prove the result for the topology $\T_D$, where $D$ is any countable dense subset of $\R^M$, we have it for all the other strong topologies. Now, for any sequence $(f_n)_\nin$ in $E$ converging to $f$ with respect to $\T_D$ and using the previously set notation, if $\{ m^j_n(\cdot)\mid\nin\}$ is $L^p_{loc}$-bounded, since by Alaoglu-Bourbaki theorem, for every $r>0$ the closed balls of $L^p([-r,r])$ are relatively compact in the weak topology $\sigma\big(L^p([-r,r]),L^q([-r,r])\big)$, then there exists a weakly convergent subsequence of $\big( m^j_{n}(\cdot)\big)_\nin$, that we keep denoting with the same indexes, with limit $m^*(\cdot)\in L^p([-r,r])$. As a consequence, the sequence of induced measures $(\mu^j_{n})_\nin$ vaguely converges to the absolutely continuous measure whose density is $m^*(\cdot)$ in $[-r,r]$. Now, recalling that  $L^p_{loc}\subset L^1_{loc}$ and using the same reasoning as before, we can obtain again $m^j(\cdot)\in L^1_{loc}$, defined as in \eqref{eq:19.05-12:37}, which satisfies \eqref{eq:21.05-12:41}, i.e. $m^j(\cdot)$ is an $m$-bound for the limit function $f$. By the uniqueness of the limit, $m^*(\cdot) $ has to coincide with $m^j(\cdot)$ in $[-r,r]$, which proves that $m^j(\cdot)\in L^p _{loc}$ and it allows to preserve the property of $L^p_{loc}$-boundedness of the $m$-bounds for $\overline E$. For the case of $L^1_{loc}$-equicontinuous $m$-bounds, the same argument used for $p=1$ applies.\par\smallskip
The result for the $l$-bounds, the $l_1$-bounds  and the $l_2$-bounds can be obtained by analogous reasoning.
\end{proof}
As we have noticed before, all the introduced topologies can be induced on $\LC(\R^{M},\R^N)$. As follows, we provide sufficient conditions under which some of those topologies coincide on suitable subsets of $\LC(\R^{M},\R^N)$. A direct consequence is that, on such sets, one can switch to the simplest and most maneuverable topologies which involve some sort of point-wise convergence. Notice also that in many applications such assumptions are trivially satisfied as, for example, when the $l$-bounds of a set $E$ are taken constant and bounded.
\par\smallskip
The following result contains Theorem 4.12 in \cite{paper:LNO1}, as a particular case.  Furthermore, under the same assumptions but using new arguments, we also extend the result of equivalence of the topologies to the weak topologies in~\eqref{eq:DelayTopologiesChain}.
\begin{thm}\label{thm:equivalence_topologies_nonhybrid}
Let $E$ be a set in $\LC_p\big(\R^{M},\R^N\big)$, $\Theta$ be any pair of suitable sets of moduli of continuity and $D$ any dense and countable subset of $\R^M$. The following statements are true.
\begin{itemize}[leftmargin=20pt]
\item[{\rm (i)}] If $E$ has $L^p_{loc}$-bounded $l$-bounds and $\T_1,\T_2$ is any pair of strong topologies in~\eqref{eq:DelayTopologiesChain}, then one has
\begin{equation*}
\!(E,\T_1)=(E,\T_2)\ \ \text{and}\ \ \mathrm{cls}_{(\SC_p(\R^{M},\R^N),\T_1)}(E)=\mathrm{cls}_{(\SC_p(\R^{M},\R^N),\T_2)}(E)\,.
\end{equation*}
\item[{\rm (ii)}] If $E$ has $L^1_{loc}$-bounded $l$-bounds and $\T_1,\T_2$ is any pair of weak topologies in~\eqref{eq:DelayTopologiesChain}, then one has
\begin{equation*}
\!(E,\T_1)=(E,\T_2)\ \ \text{and}\ \ \mathrm{cls}_{(\SC(\R^{M},\R^N),\T_1)}(E)=\mathrm{cls}_{(\SC(\R^{M},\R^N),\T_2)}(E)\,.
\end{equation*}
\end{itemize}
\end{thm}
\begin{proof}
The statement (i) is a consequence of Theorem 4.12 of~\cite{paper:LNO1} and the order relation of the strong topologies in~\eqref{eq:DelayTopologiesChain}. \par\smallskip
As regards (ii), fix $D$, dense and countable subset of $\R^M$, and  $\Theta$, suitable set of moduli of continuity as in Definition \ref{def:ssmc}. Thanks to Proposition~\ref{prop:07.07-19:44p=1}, we know that $\mathrm{cls}_{(\SC(\R^M,\R^N),\sigma_D)}(E)\subset\LC(\R^M)$ and it has  $L^1_{loc}$-bounded $l$-bounds. Assume that $(f_n)_\nin$ is a sequence of elements in $ E$ converging to some $f$ in $\big(\SC(\R^{M},\R^N),\sigma_{D}\big)$. We prove that $(f_n)_\nin$ converges to $f$ in $\big(\SC(\R^{M},\R^N),\sigma_{\Theta}\big)$. We proceed by dividing the proof in two steps.\par\smallskip
\emph{Step 1.} Consider a set $E_1$ with $L^1_{loc}$-bounded $m$-bounds and $L^1_{loc}$-bounded $l$-bounds. Let $(h_n)_\nin$ be a sequence of elements of $E_1$ converging to some $h$ in $\big(\LC(\R^{M},\R^N),\sigma_{ D}\big)$. We shall prove the convergence  in $\big(\LC(\R^{M},\R^N),\sigma_{\Theta}\big)$. Fix a compact interval  $I=[q_1,q_2]$, with $q_1,q_2\in\Q$, $j\in\N$ and, for any $\nin$,  let $m_n^j(\cdot),l_n^j(\cdot)\in L^1_{loc}$ be respectively the optimal $m$-bound and the optimal $l$-bound of $h_n$ on $B_j$. By the $L^1_{loc}$-boundedness of the $l$-bounds, there is a $\rho>0$ such that
\begin{equation*}
\sup_{\nin}\int_Il^j_n(s)\, ds<\rho<\infty\, .
\end{equation*}
Fix $\ep>0$ and consider $\delta=\ep/3\rho$. Since $ B_j \subset \R^{M}$ is compact, and $D$ is dense in $\R^{M}$, there exist $x_1,\dots x_{\nu}\in D$ such that $B_j\subset \bigcup_{i=1}^{\nu} \overcircle{B_\delta}(x_i)$, where $\overcircle{B_\delta}(x)$ 
denotes the open ball of $\R^{M}$ of radius $\delta$ centered at $x\in\R^{M}$. For $i=1,\dots,\nu$, let us consider the continuous functions $\phi_i:\R^{M}\to[0,1]$, so that
\begin{equation*}
\supp(\phi_i)\subset \overcircle{B_{\delta}}(x_i)\, \qquad \mathrm{and}\qquad \sum_{i=1}^\nu \phi_i(x)=1 \quad \forall\, x\in B_{j} \, ,
\end{equation*}
and define the functions
\begin{equation*}
h^*_n(t,x)=\sum_{i=1}^\nu\phi_i(x)\, h_n(t,x_i)\qquad \mathrm{and}\qquad
h^*(t,x)=\sum_{i=1}^\nu \phi_i(x)\,  h(t,x_i)\, .
\end{equation*}
Denoted by $\K^I_j$ the compact subset of  $C(I,B_j)$ admitting $\theta^
I_j\in\Theta$ as a modulus of continuity,  one has that for any $x(\cdot)\in \K^I_j$
\begin{equation}\label{eq:25.05-16:26}
\begin{split}
\!\!\!\bigg|\int_I&\big[h_{n}\big(t,x(t)\big)-h\big(t,x(t)\big)\big]dt\bigg|\le \bigg|\int_I\!\big[h_{n}\big(t,x(t)\big)- h^*_{n}\big(t,x(t)\big)\big]dt\bigg| \\
&\ +\bigg|\int_I\!\big[h^*_{n}\big(t,x(t)\big)-h^*\big(t,x(t)\big)\big]dt\bigg|+\bigg|\int_I\!\big[h^*\big(t,x(t)\big)-h\big(t,x(t)\big)\big]dt\bigg| \, .
\end{split}
\end{equation}
Let us separately analyze each element in the sum on the right-hand side of equation~\eqref{eq:25.05-16:26}. As regards the first one, we have that
\begin{equation}\label{eq:03.05-17:38}
\begin{split}
\bigg|\int_I\big[h_{n}\big(t,&\,x(t)\big)- h^*_{n}\big(t,x(t)\big)\big]dt\bigg|\\
&=\bigg|\int_I\sum_{i=1}^\nu \phi_i\big(x(t)\big)\left[ h_{n}\big(t,x(t)\big)- h_{n}(t,x_i)\right]dt\bigg|\\
&\le \int_I \sum_{i=1}^\nu \phi_i\big(x(t)\big)\, l^j_{n}(t)\ \delta\, dt=\frac{\ep}{3\rho}\int_Il^j_{n}(t)\, dt\le \frac{\ep}{3}\, .
\end{split}
\end{equation}
Similar reasonings apply to the third element of the sum in \eqref{eq:25.05-16:26}: in particular, recall that, reasoning as in~\eqref{eq:boundedMbounds} and \eqref{eq:boundedMbounds-sameCONSTANT}, the $l$-bound for $h$ on $B_j$, namely $\bar{l}^j(\cdot)\in L^1_{loc}$,  satisfies
\begin{equation*}
\int_I\bar l^j(s)\, ds<\rho<\infty.
\end{equation*}
Finally, let us deal with  the remaining integral in \eqref{eq:25.05-16:26}.  By the uniform continuity of the functions $\phi_i(\cdot)$ on $B_j$, and recalling that all $x(\cdot)\in\K^I_j$ share the same modulus of continuity, we have that for the given $\ep>0$ there exists $\delta>0$ such that for all $i\in\{1,\dots,\nu\}$ one has
\begin{equation*}
\forall \,s,t\in I,\, \forall\, x(\cdot)\in\K^I_j:\quad |s-t|<\delta\quad \Rightarrow\quad  \big|\phi_i\big(x(s)\big)-\phi_i\big(x(t)\big)\big|<\frac{\ep}{9\, \nu\rho_m}\,,
\end{equation*}
where
\begin{equation*}
\rho_m:=\max\left\{\int_Im^j(t)\,dt,\,\sup_\nin\int_Im^j_{n}(t)\,dt\right\}<\infty\,,
\end{equation*}
and $m_j(\cdot)\in L^1_{loc}$  denotes the optimal $m$-bound for $h$ on $B^j$ whose existence is guaranteed by Proposition~\ref{prop:07.07-19:44p=1}. In particular, the constant $\rho_m$ is well defined thanks to the $L^1_{loc}$-boundedness of  the $m$-bounds of the functions in $E_1$.
Thus, let us consider a $\delta$-partition of $I$, i.e. $\tau_1,\dots,\tau_{\eta}\in I\cap\Q$ such that $I=[\tau_1,\tau_\eta]$ and $0<\tau_{k+1}-\tau_k<\delta$, for any $k=1,\dots,\eta-1$, and a function
\begin{equation*}
\bar \phi_i: \K^I_j\to L^\infty(I,\R)\quad\text{ defined by }\quad
\bar \phi_i(x)(t)= \sum_{k=1}^{\eta} \phi_i\big(x(\tau_k)\big)\chi_{(\tau_k,\tau_{k+1}]}(t)
\end{equation*}
Notice that, for any $x(\cdot)\in\K^I_j$ and any $i=1,\dots,\nu$ one has
\begin{equation*}
\|\phi_i\big(x(\cdot)\big)-\bar \phi_i(x)(\cdot)\|_{L^\infty(I)}<\frac{\ep}{9\,\nu\rho_m}\,,
\end{equation*}
Now, we can write
\begin{equation}\label{eq:02.05-16:41}
\begin{split}
\bigg|\int_I\big[h^*_{n}\big(t,&\,x(t)\big)-h^*\big(t,x(t)\big)\big]dt\bigg|\\
&\le\sum_{i=1}^{\nu}\bigg|\int_I \phi_i\big(x(t)\big)\left[ h_{n}\big(t,x_i\big)- h(t,x_i)\right]dt\bigg|\\
&\le \sum_{i=1}^{\nu}\bigg|\int_I \bar \phi_i(x)(t)\left[ h_{n}\big(t,x_i\big)- h(t,x_i)\right]dt\bigg|\\
&\qquad+\sum_{i=1}^{\nu}\int_I \big| h_{n}\big(t,x_i\big)\big|\ \big|\phi_i\big(x(t)\big)- \bar \phi_i(x)(t)\big|\,dt\\
&\qquad+\sum_{i=1}^{\nu}\int_I \big| h\big(t,x_i\big)\big|\ \big|\bar \phi_i(x)(t)- \phi_i\big(x(t)\big)\big|\,dt\\
&\le \sum_{i=1}^{\nu}\bigg[\,\sum_{k=1}^{\eta}\phi_i\big(x(\tau_k)\big)\left|\int_{\tau_k}^{\tau_{k+1}}\left[ h_{n}\big(t,x_i\big)- h(t,x_i)\right]dt\right|\\
&\qquad+2\,\rho_m\|\phi_i\big(x(\cdot)\big)-\bar \phi_i(x)(\cdot)\|_{L^\infty(I)}\bigg].
\end{split}
\end{equation}
By the convergence  of $(h_n)_\nin$ to $h$ in $\big(\LC(\R^{M},\R^N),\sigma_{D}\big)$ and considering that we are only using a finite number of points $x_i$, with $i=1,\dots,\nu$, there exists $n_0\in\N$ such that, if $n>n_0$, then for all $i=1,\dots,\nu$ and for all $k=1,\dots,\eta$ one has
\begin{equation*}
\left|\int_{\tau_k}^{\tau_{k+1}}\left[ h_{n}\big(t,x_i\big)-h(t,x_i)\right]dt\right|<\frac{\ep}{9\,\nu\,\eta}\,.
\end{equation*}
Additionally, notice that due to the convergence of $(h_n)_\nin$ to $h$ with respect to $\sigma_{D}$, inequality~\eqref{eq:02.05-16:41} is uniform in $x(\cdot)\in\K^I_j$.
Thus, for $n>n_0$, \eqref{eq:02.05-16:41} becomes
\begin{equation}\label{eq:02.05-16:40}
\bigg|\int_I\big[h^*_{n}\big(t,x(t)\big)-h^*\big(t,x(t)\big)\big]dt\bigg| <\frac{\ep}{9} +\frac{2\,\nu\rho_m\ep}{9\,\nu\rho_m}=\frac{\ep}{3}.
\end{equation}
From \eqref{eq:25.05-16:26}, \eqref{eq:03.05-17:38} and \eqref{eq:02.05-16:40}, and the uniformity in $x(\cdot)\in\K^I_j$, we obtain that  the sequence $(h_n)_\nin$ converges to $h$ in $\big(\LC(\R^{M},\R^N),\sigma_{\Theta}\big)$. Consequently, the topologies of type $\sigma_D$ and $\sigma_\Theta$ are equivalent on $E_1$.\par\smallskip
\emph{Step 2.} Consider $x_0\in B_1\cap D$ and define the functions
\begin{equation*}
h(t,x)= f(t,x)-f(t,x_0)\quad\text{and}\quad h_n(t,x)= f_n(t,x)-f_n(t,x_0),\ \forall\,\nin.
\end{equation*}
Using the $L^1_{loc}$-boundedness of the $l$-bounds of $E$, it is easy to show that the set $\{h_n\mid\nin\}\cup\{h\}$ has $L^1_{loc}$-bounded $m$-bounds and $L^1_{loc}$-bounded $l$-bounds. Furthermore, from the convergence of $(f_n)_\nin$ to $f$ in $\big(\SC(\R^{M},\R^N),\sigma_{D}\big)$, one easily deduces that also $(h_n)_\nin$ converges to $h$ in $\big(\SC(\R^{M},\R^N),\sigma_{ D}\big)$. Therefore, the assumptions of step 1 apply to the set $E_1= \{h_n\mid\nin\}\cup\{h\}$ and thus one has that $(h_n)_\nin$ converges to $h$ in $\big(\LC(\R^{M}),\sigma_{ \Theta}\big)$. Hence, for each interval $I=[q_1,q_2]$, with $q_1,q_2\in\Q$ and for each $j\in\N$ one has
\begin{equation*}
\begin{split}
\sup_{x(\cdot)\in\K^I_j}\bigg|&\int_I[f_n\big(t,x(t)\big)-f\big(t,x(t)\big)]\,dt\bigg|\\
&\le \sup_{x(\cdot)\in\K^I_j}\bigg|\int_I[h_n\big(t,x(t)\big)-h\big(t,x(t)\big)]\,dt\bigg|+\bigg|\int_I[f_n\big(t,x_0\big)-f\big(t,x_0\big)]\,dt\bigg|,
\end{split}
\end{equation*}
and the right-hand side goes to zero as $\nti$ because $(h_n)_\nin$ converges to $h$ in $\big(\LC(\R^{M}),\sigma_{ \Theta}\big)$ and $(f_n)_\nin$ converges to $f$ in $\big(\LC(\R^{M}),\sigma_{ D}\big)$, which implies that $(f_n)_\nin$ converges to $f$ in $\big(\LC(\R^{M}),\sigma_{ \Theta }\big)$ and, as a consequence, all the topologies of type $\sigma_D$ and $\sigma_\Theta$ coincide on $E$.
\end{proof}
On the other hand, by weakening the assumptions of Theorem \ref{thm:equivalence_topologies_nonhybrid} to the sole $l_2$-bounds, we show how it is still possible to obtain equivalence of the hybrid topologies (either weak or strong).
\begin{thm}\label{thm:equivalence_topologies}
Let $E$ be a set in $\SC_p\big(\R^{2N},\R^N\big)$, $\Theta$ and $\widehat\Theta$ be any pair of suitable sets of moduli of continuity and $D$ any dense and countable subset of $\R^N$. The following statements are true.
\begin{itemize}[leftmargin=20pt]
\item[{\rm (i)}]  If $E$ has $L^p_{loc}$-bounded $l_2$-bounds and $\T_1,\T_2\in\{ \T_{\Theta D},\T_{\Theta \smash{\widehat\Theta}},\T_{\Theta B}\}$, then one has
\begin{equation*}
\!(E,\T_1)=(E,\T_2)\ \ \text{and}\ \ \mathrm{cls}_{(\SC_p(\R^{2N},\R^N),\T_1)}(E)=\mathrm{cls}_{(\SC_p(\R^{2N},\R^N),\T_2)}(E).
\end{equation*}
\item[{\rm (ii)}]  If $E$ has $L^1_{loc}$-bounded $l_2$-bounds and $\T_1,\T_2\in\{ \sigma_{\Theta D},\sigma_{\Theta \smash{\widehat\Theta}}\}$, then one has
\begin{equation*}
\!(E,\T_1)=(E,\T_2)\ \ \text{and}\ \ \mathrm{cls}_{(\SC(\R^{2N},\R^N),\T_1)}(E)=\mathrm{cls}_{(\SC(\R^{2N},\R^N),\T_2)}(E).
\end{equation*}
\end{itemize}
\end{thm}
\begin{proof}
We will prove (ii), that is, the case of the weak hybrid topologies because (i) can be carried out with the same reasoning and simplifying some arguments. Consider $E\subset\SC\big(\R^{2N},\R^N\big)$ with $L^1_{loc}$-bounded $l_2$-bounds, and fixed any numerable set $D$, dense in $\R^{N}$, and any pair of suitable sets of moduli of continuity $\Theta$ and $\smash{\widehat\Theta}$ (see {\rm Definition~\ref{def:ssmc}}), assume that $(f_n)_\nin$ is a sequence of elements of $ E$ converging to some $f$ in $\big(\SC(\R^{2N},\R^N),\sigma_{\Theta D}\big)$. We shall prove that $(f_n)_\nin$ converges to $f$ in $\big(\SC(\R^{2N},\R^N),\sigma_{\Theta \smash{\widehat\Theta}}\big)$. In order to simplify the notation,
we use the following symbols:
\begin{equation*}
\overline E_{\Theta D}=\mathrm{cls}_{(\SC(\R^{2N},\,\R^N),\sigma_{\Theta D})}(E),\quad\text{and}\quad \overline E_{\Theta \smash{\widehat\Theta}}=\mathrm{cls}_{(\SC(\R^{2N},\,\R^N),\sigma_{\Theta \smash{\widehat\Theta}})}(E).
\end{equation*}
Notice that, thanks to Proposition~\ref{prop:07.07-19:44p=1}, both $\overline E_{\Theta D}$ and $ \overline E_{\Theta \smash{\widehat\Theta}}$ have $L^1_{loc}$-bounded $l_2$-bounds. Thus, fixed $x(\cdot)\in \K^I_j$, we can define the functions $\widehat f,\widehat f_n\in\LC(\R^{N},\R^N)$, with $\nin$, by
\begin{equation*}
\widehat f (t,u)=f\big(t,x(t ),u\big)\quad\text{and}\quad \widehat f_n (t,u)=f_n\big(t,x(t ),u\big)\,.
\end{equation*}
It is immediate to prove that  $(\widehat f_n)_\nin$ converges $\widehat f$ in $\big(\LC(\R^{N},\R^N),\sigma_{D}\big)$. In fact, since, by construction $\{\widehat f_n\mid\nin\}\cup\{\widehat f\}$ has $L^1_{loc}$-bounded $l$-bounds, then, due to Theorem \ref{thm:equivalence_topologies_nonhybrid}(ii), we have that $(\widehat f_n)_\nin$ converges to $\widehat f$ in $\big(\LC(\R^{N},\R^N),\sigma_{\smash{\widehat\Theta}}\big)$. Consequently, for any interval $I$ with rational extrema and $j\in\N$ one has
\begin{equation}\label{eq:09.11:19:12}
\sup_{u(\cdot)\in\widehat\K^I_j}\bigg|\int_I[f_n\big(t,x(t),u(t)\big)-f\big(t,x(t),u(t)\big)]\,dt\bigg|\xrightarrow{\nti} 0\,.
\end{equation}
As a matter of fact, \eqref{eq:09.11:19:12} is also uniform in $x(\cdot)\in \K^I_j$. Indeed, if this were not the case there would exist an $\ep>0$, a subsequence $(n_k)_{k\in\N}$, and a sequence $\big(x_k(\cdot)\big)_{k\in\N}$ in $\K_j^I$ such that
\begin{equation}\label{eq:09.11-19:10}
\ep<\sup_{u(\cdot)\in\widehat\K^I_j}\bigg|\int_I[f_{n_k}\big(t,x_k(t),u(t)\big) -f\big(t,x_k(t),u(t)\big)]\,dt\bigg|\,.
\end{equation}
However, if we consider the functions $g,g_k\in\LC(\R^{N},\R^N)$, with $k\in\N$, defined by
\begin{equation*}
g(t,u)=f\big(t,x(t ),u\big)\quad\text{and}\quad g_k (t,u)=f_{n_k}\big(t,x_k(t ),u\big)\,,
\end{equation*}
where $x(\cdot)$ is the limit, up to a subsequence, of $\big(x_{n_k}(\cdot)\big)_{k\in\N}$ (such limit exists because $\K_j^I$ is compact), we have that $(g_{k})_{k\in\N}$ converges to $g$ in $\big(\LC(\R^{N},\R^N),\sigma_{D}\big)$. Indeed, for any interval $I$ with rational extrema and any $u\in D$ one has
\begin{align*}
\bigg|\int_I[g_k(t,u)-g(t,u)]\,dt\bigg|\le & \bigg|\int_I[f_{n_k}\big(t,x_k(t),u\big)-f\big(t,x_k(t),u\big)]\,dt\bigg|\\
& \quad + \bigg|\int_I[f\big(t,x_k(t),u\big)-f\big(t,x(t),u\big)]\,dt\bigg|\,,
\end{align*}
and the right-hand side of the previous inequality goes to zero as $k\to\infty$ because by assumption $(f_n)_\nin$ converges to $f$ in $\big(\SC(\R^{2N},\R^N),\sigma_{\Theta D}\big)$ and due to the fact that $f\in\SC(\R^{2N},\R^N)$ and $\big(x_k(\cdot)\big)_{k\in\N}$ converges to $x(\cdot)$. Hence, noticing that $\{g_k\mid k\in\N\}\cup\{g\}$ has $L^1_{loc}$-bounded $l$-bound,  from Theorem~\ref{thm:equivalence_topologies_nonhybrid}(ii) one has that $(g_{k})_{k\in\N}$ converges to $g$ in $\big(\LC(\R^{N},\R^N),\sigma_{\smash{\widehat\Theta}}\big)$ which in turn contradicts \eqref{eq:09.11-19:10}. Therefore, we have that \eqref{eq:09.11:19:12} is uniform in $x(\cdot)\in \K^I_j$, which, thanks to Lemma~\ref{lem:alternative-seminorms}, shows that $(f_n)_\nin$ converges to $f$ in $\big(\SC(\R^{2N},\R^N),\sigma_{\Theta \smash{\widehat\Theta}}\big)$ and ends the proof.
\end{proof}
\section{Continuous dependence of the solutions in $\mathcal{C}$}\label{sec:Continuity-classic-topologies}
This section contains the results of continuous variation on initial data for the topologies treated in Section~\ref{sectopo} and the respective results of continuity of the skew-product semiflows on $E\times C([-1,0],\R^N)$ generated by the delay differential problem of the type \eqref{eq:16.07-21:13}, where $E\subset\LC(\R^{2N},\R^N)$ satisfies specific assumptions on the $m$-bounds and/or $l$-bounds and $ C([-1,0],\R^N)$ denotes the set of continuous functions mapping $[-1,0]$ onto $\R^N$. \par\smallskip
As a general rule, when a result of equivalence of the topologies (either Theorem~\ref{thm:equivalence_topologies_nonhybrid} or Theorem~\ref{thm:equivalence_topologies}) holds on $E$, we will write the continuity of the skew-product semiflow using the most practical available topology, i.e. the one that involves some kind of point-wise convergence. It is implicit that the same result holds with respect to any of the other equivalent or stronger topologies of type~\eqref{eq:DelayTopologiesChain}.\par\smallskip
Throughout the whole section (and the rest of the work), an important role will be played by specific suitable sets of moduli of continuity, that is the ones provided by the $m$-bounds of a Carath\'eodory function or of a set of Carath\'eodory functions.
\begin{defn}\label{def:THETAjseq}
Let $E\subset\LC_p(\R^{2N},\R^N)$ admit $L^1_{loc}$-equicontinuous $m$-bounds. For any $j\in\N$ and for any interval $I=[q_1,q_2]$, $q_1,q_2\in\Q$, define
\begin{equation*}
\theta^I_j(s):=
 \sup_{t\in I,f\in E}\int_t^{t+s}m_f^j(u)\, du\, ,
\end{equation*}
where, for any $f\in E$, the function $m_f^j(\cdot)\in L^1_{loc}$ denotes the optimal $m$-bounds of $f$ on $B_j$. Notice that, since $E$ admits $L^1_{loc}$-equicontinuous $m$-bounds, then  $\Theta=\{\theta^I_j(\cdot)\mid I=[q_1,q_2],\, q_1,q_2\in\Q,\, j\in\N\}$ defines a suitable set of moduli of continuity.
\end{defn}
\begin{rmk}\label{rmk:THETAjfunc}
If $f\in\LC_p(\R^{2N},\R^N)$ has $L^1_{loc}$-equicontinuous $m$-bounds we similarly define for any $B_j\subset\R^N$,
\begin{equation}\label{eq:moduli-m-bounds-f}
\theta_j(s):= \sup_{t\in\R}\int_t^{t+s}m^j(u)\, du\, ,
\end{equation}
where $m^j(\cdot)$ is the optimal $m$-bound for $f$ on $B_j$. Here again, notice that $\Theta=\{\theta^I_j(\cdot)\mid I=[q_1,q_2],\, q_1,q_2\in\Q,\, j\in\N\}$  defines a suitable set of moduli of continuity thanks to the $L^1_{loc}$-equicontinuity. As a matter of fact, the so-defined suitable set of moduli of continuity $\Theta$, is the one defined in Definition~\ref{def:THETAjseq} when $E=\{f_t\mid t\in\R\}$ that is, $E$ is the set of the time translations of $f$. For this reason, the elements of $\Theta$ are independent of the interval $I$. In other words, in this case, for any $j\in\N$ one has that any $\theta^I_j(\cdot)$ of Definition~\ref{def:THETAjseq}, where $I=[q_1,q_2]$, $q_1,q_2\in\Q$,  is in fact the  $\theta_j(\cdot)$ of \eqref{eq:moduli-m-bounds-f}.
\end{rmk}
For the sake of completeness and to set some notation, we also state a theorem of existence and uniqueness of the solution for a Cauchy Problem of Carath\'eodory type with constant delay.  A proof can be derived by the one given for Carath\'eodory ordinary differential equations in Coddington and Levinson~\cite[Theorem 1.1, p.43, Theorem 1.2, p.45 and Theorem 2.2, p.49]{book:CL} (see Hale and Cruz~\cite{paper:CH} and Hale and Verduyn Lunel~\cite{book:HAL} for the proofs of existence, uniqueness and continuous dependence for a more general class of delay differential equations). Notice also that, in order to simplify the notation, from now on we will denote by $\mathcal{C}$ the set
\[
\mathcal{C}:= C([-1,0],\R^N).
\]
\begin{thm}\label{thm:13.04-16:49}
For any $f\in\LC_p(\R^{2N},\R^N)$ and any $\phi\in \mathcal{C}$ there exists a maximal interval $I_{f,\phi}=[-1,b_{f,\phi})$ and a unique continuous function $x(\cdot,f,\phi)$ defined on $I_{f,\phi}$ which is the solution of the delay differential problem
\begin{equation}\label{eq:solLCEDDE}
\begin{cases}
\dot x=f\big(t,x(t),x(t-1)\big)&\text{for }t>0\,,\\
x(t)=\phi(t) &\text{for }t\in[-1,0].
\end{cases}
\end{equation}
In particular, if $b_{f,\phi}<\infty$, then $|x(t,f,\phi)|\to\infty$ as $t\to b_{f,\phi}$.
\end{thm}
\subsection{Continuity with respect to $\T_B$, $\T_D$ and $\sigma_D$}\label{subsec:classictopologies}
As follows, we prove the first results of continuous variation on initial data for solutions of problems like \eqref{eq:solLCEDDE} and deduce the continuity of the generated skew-product semiflows with respect to the classic topologies $\T_B$, $\T_D$ and $\sigma_D$. The statement show a parallelism with the analogous results in \cite{paper:LNO1} and \cite{paper:LNO2} for Carath\'eodory ordinary differential equations.
\begin{thm}\label{thm:ContinuousTBdelay}
Consider a sequence $(f_n)_\nin$ in $\LC_p(\R^{2N},\R^N)$ converging to $f$ in $(\LC_p(\R^{2N},\R^N),\T_B)$ and $\big(\phi_{n}(\cdot)\big)_\nin$ in $\mathcal{C}$ converging uniformly to $\phi\in \mathcal{C}$. Then, with the notation of {\rm Theorem~\ref{thm:13.04-16:49}},
\begin{equation*}
 x(\cdot,f_n,\phi_{n}) \xrightarrow{\nti}  x(\cdot,f,\phi)
\end{equation*}
uniformly in any $[-1,T]\subset I_{f,\phi}$.
\end{thm}
\begin{proof}
Following the notation of Theorem~\ref{thm:13.04-16:49}, let $I_{f,\phi}=[-1,b_{f,\phi})$ be the maximal interval of definition of the solution of the delay differential problem~\eqref{eq:solLCEDDE}.  For $n\in\N$ we define the family of functions belonging to $\LC_p(\R^{N},\R^N)$
\begin{equation}\label{eq:gng}
\begin{split}
g_n(t,x)&:=
\begin{cases}
f_n\big(t,x,\phi_n(t-1)\big), &\text{if }t\in[0,1]\,,\\
0 &\text{otherwise}\,,
\end{cases}\\
g(t,x)&:=
\begin{cases}
f\big(t,x,\phi(t-1)\big), &\;\;\text{ if }t\in[0,1]\,,\\
0 &\;\;\text{ otherwise}.
\end{cases}
\end{split}
\end{equation}
 Moreover, for any $I=[q_1,q_2]$, $q_1,q_2\in\Q$ and $j\in\N$ one has
\begin{align*}
&\sup_{x(\cdot)\in C(I,B_j)}\bigg[\int_I\left| g_n\big(t,x(t)\big)-g\big(t,x(t)\big) \right|^pdt\bigg]^{1/p}\\
& \le\sup_{x(\cdot)\in C([0,1],B_j)}\bigg[\int_0^1 \left| f_n\big(t,x(t),\phi_n(t-1)\big)-f\big(t,x(t),\phi_n(t-1)\big)\right|^pdt\bigg]^{1/p}\\
& \quad + \sup_{x(\cdot)\in C([0,1],B_j)}\bigg[\int_0^1 \left| f\big(t,x(t),\phi_n(t-1)\big)-f\big(t,x(t),\phi(t-1)\big)\right|^pdt\bigg]^{1/p}\!\!\!\!\!\!=  I_1+ I_2\,.
\end{align*}
Hence, for an appropriate $k\geq j$, we deduce that
\begin{equation}\label{eq:05.10-13-31}
I_1\leq \sup_{(x(\cdot),y(\cdot))\in C([0,1],B_k)}\bigg[\int_0^1  \left| f_n\big(t,x(t),y(t)\big)-f\big(t,x(t),y(t)\big)\right|^pdt\bigg]^{1/p},
\end{equation}
which goes to $0$ as $\nti$ because $(f_n)_\nin$  converges to $f$ in $(\LC_p(\R^{2N},\R^N),\T_B)$, and that
\begin{equation*}\label{eq:05.10-13-32}
I_2\leq\bigg[ \int_0^1 \Big( l^k(t)\Big)^p \left| \phi_n(t-1)-\phi(t-1)\right|^pdt\bigg]^{1/p},
\end{equation*}
which goes to $0$ as $\nti$ due to the uniform convergence of $(\phi_n)_\nin$ to $\phi$ in $\mathcal{C}$, and where $l^k(\cdot)\in L^1_{loc}$ is the optimal $l$-bound for $f$ on $B_k\subset \R^{2N}$. Therefore  we have that $(g_n)_{n\in\N}$ converges to $g$ in
 $(\LC_p(\R^{N},\R^N),\T_B)$,   and since $x(t,f_n,\phi_n)$ and $x(t,f,\phi)$ are respectively the solutions of
\begin{equation}\label{eq:0510-12:26}
\begin{cases}
\dot x=g_n\big(t,x\big)\,,\\
x(0)=\phi_n(0)\,,
\end{cases}
\text{ and }\quad\begin{cases}
\dot x=g\big(t,x\big)\,,\\
x(0)=\phi(0)\,,
\end{cases}
\end{equation}
on $[0,1]$, by Theorem~4 in~\cite{paper:RMGS1} we conclude that $x(\cdot,f_n,\phi_{n}) \xrightarrow{\nti}  x(\cdot,f,\phi)$ on $[0,T]$, where $0\le T<\min\{1,b_{f,\phi}\}$. If $b_{f,\phi}\leq 1$, the above reasoning shows the convergence of the solutions on any compact subset of $[-1,b_{f,\phi})$. On the other hand, if $b_{f,\phi}>1$, the previous reasoning shows the convergence of the solutions on $[-1,1]$ and  the  recursive iteration of such argument provides the uniform convergence on any compact subset $[-1,T]$ of the maximal interval $I_{f,\phi}$, which ends the proof.
\end{proof}
As a direct consequence of Theorem~\ref{thm:ContinuousTBdelay} and Theorem~\ref{thm:equivalence_topologies_nonhybrid}(i), one can obtain an analogous result for the topology $\T_D$, whose proof is omitted.
\begin{thm}\label{thm:ContinuousTDdelay}
Consider a set $E\subset\LC_p(\R^{2N},\R^N)$ with $L^1_{loc}$-bounded $l$-bounds, a sequence $(f_n)_\nin$ in $E$ converging to $f$ in $(\LC_p(\R^{2N},\R^N),\T_D)$ and  a sequence $\big(\phi_{n}(\cdot)\big)_\nin$  in $\mathcal{C}$ converging uniformly to $\phi\in \mathcal{C}$. Then, with the notation of {\rm Theorem~\ref{thm:13.04-16:49}}, one has
\begin{equation*}
 x(\cdot,f_n,\phi_{n}) \xrightarrow{\nti}  x(\cdot,f,\phi)
\end{equation*}
uniformly in any $[-1,T]\subset I_{f,\phi}$.
\end{thm}
As regards the topology $\sigma_D$, the idea is, once again, to use Theorem~\ref{thm:equivalence_topologies_nonhybrid}(ii) so that for any specific converging sequence of initial data, one selects the appropriate suitable set of moduli of continuity as shown in the following result which also permits to simplify the proof of Theorem 3.8 in \cite{paper:LNO2}.
\begin{thm}\label{thm:ContinuousSigmaDdelay}
Consider $E\subset\LC(\R^{2N},\R^N)$  with $L^1_{loc}$-equicontinuous $m$-bounds and let $\Theta$ be the suitable set of moduli of continuity given by the $m$-bounds.
\begin{itemize}[leftmargin=20pt]\setlength\itemsep{0.3em}
\item[\rm (i)]   Let $\big(\phi_{n}(\cdot)\big)_\nin$ be a sequence in $\mathcal{C}$ converging uniformly to $\phi\in \mathcal{C}$ and let $\theta_0$ be the modulus of continuity shared by all the functions in $\{\phi_n\mid\nin\}\cup\{\phi\}$ (such a $\theta_0$ exists thanks to Ascoli-Arzel\'a's theorem). Furthermore, let  $\overline\Theta$ be the suitable set of moduli of continuity whose elements are constructed as follows: for any $I=[q_1,q_2]$, $q_1,q_2\in\Q$ and $j\in\N$
\begin{equation*}
\bar\theta^I_j(s)=\max\{\theta^I_j(s) ,\,\theta_0(s)\}.
\end{equation*}
If $(f_n)_\nin$ is a sequence in $E$ converging to $f$ in $(\LC(\R^{2N},\R^N),\sigma_{\smash{\Theta\overline\Theta}})$, then, with the notation of {\rm Theorem~\ref{thm:13.04-16:49}}, one has that
\begin{equation*}
 x(\cdot,f_n,\phi_{n}) \xrightarrow{\nti}  x(\cdot,f,\phi)
\end{equation*}
uniformly in any $[-1,T]\subset I_{f,\phi}$.
\item[\rm (ii)] If additionally $E$ has $L^1_{loc}$-bounded $l$-bounds, then the result keeps holding true when $(f_n)_\nin$ is a sequence in $E$ converging to $f$ in $(\LC(\R^{2N},\R^N),\sigma_D)$.
\end{itemize}
\end{thm}
\begin{proof}
The proof  of (i) can be carried out with a reasoning similar to the one used in the proof of Theorem~\ref{thm:ContinuousTBdelay}. Again,  define $g$ and $g_n$, with $\nin$, as in \eqref{eq:gng}.  Let $\overline K_i^I$ be the set of functions in $C(I,B_j)$ which admit $\overline\theta^I_j$ as a modulus of continuity.  Consider $i\in\N$, such that $\{\phi_n(\cdot)\mid\nin\}\cup\{\phi\}\subset\overline\K^{[-1,0]}_i$,  $I=[q_1,q_2]$, $q_1,q_2\in\Q$, $j\in\N$, and let $k\in\N$ be an integer such that $k\ge \max\{i,j\}$. Denoting by $l^{2k}(\cdot)\in L^1_{loc}$ the optimal $l$-bound for $f$ on $B_{2k}\subset \R^{2N}$  and using the triangular inequality, one has that
\begin{equation}\label{eq:05.10-11:41}
\begin{split}
\!\!\!\sup_{x(\cdot)\in \K^{I}_j}\left|\int_I\big[ g_n\big(t,x(t)\big)-g\big(t,x(t)\big)\big] \,dt\right|\le \int_0^1  l^{2k}(t) \left| \phi_n(t-1)-\phi(t-1)\right|\,dt& \,+\\[-0.3ex]
 +\sup_{\substack{x(\cdot)\in \K^{I}_k,\\[0.5ex] y(\cdot) \in \overline\K^{[-1,0]}_{k}}}\bigg| \int_{I\cap[0,1]} \big[ f_n\big(t,x(t),y(t-1)\big)-f\big(t,x(t),y(t-1)\big)\big]\,dt&\,\bigg|.
\end{split}
\end{equation}
In particular, the first integral on the right-hand side of the previous inequality goes to $0$ as $\nti$ due to the uniform convergence of $(\phi_n)_\nin$ to $\phi$ in $\mathcal{C}$.  As regards the second integral, notice that the case $I\cap[0,1]=\varnothing$ is trivial; on the other hand, if $I\cap[0,1]$ is nonempty, one can consider an interval $J$ with rational extrema such that $I\cup[0,1]\subset J$ and, up to extending the functions in $\K^{I}_k$ and $ \overline\K^{[-1,0]}_{k}$ by constants to $J$ and $J-1$ respectively, one has that, since $(f_n)_\nin$ converges to $f$ in $(\LC(\R^{2N},\R^N),\sigma_{\smash{\Theta\overline\Theta}})$, then, thanks to Lemma~\ref{lem:conv-subintHYBRID},
\[
\lim_\nti\sup_{\substack{x(\cdot)\in \K^{J}_k,\\[0.5ex] y(\cdot) \in \overline\K^{J-1}_{k}}}\bigg| \int_{I\cap[0,1]} \big[ f_n\big(t,x(t),y(t-1)\big)-f\big(t,x(t),y(t-1)\big)\big]\,dt\bigg|=0.
\]
Consequently, also the second integral at the right-hand side of \eqref{eq:05.10-11:41} goes to zero as $\nti$.
\par\smallskip
Therefore  we have that $(g_n)_{n\in\N}$ converges to $g$ in $(\LC(\R^{N},\R^N),\sigma_{\Theta})$, and since $x(t,f_n,\phi_n)$ and $x(t,f,\phi)$ are respectively the solutions of the systems in \eqref{eq:0510-12:26}, then, applying Theorem 3.8 in \cite{paper:LNO2}, we have that $x(\cdot,f_n,\phi_{n}) \xrightarrow{\nti}  x(\cdot,f,\phi)$  on $[0,T]$, where $0\le T<\min\{1,b_{f,\phi}\}$. If $b_{f,\phi}\leq 1$ the above reasoning shows the convergence of the solutions on any compact subset of $[-1,b_{f,\phi})$. On the other hand, if $b_{f,\phi}>1$, the previous reasoning shows the convergence of the solutions on $[-1,1]$ and  the  recursive iteration of such argument provides the uniform convergence on any compact subset $[-1,T]$ of the maximal interval $I_{f,\phi}$, which finishes the proof of (i).\par\smallskip
The proof of (ii) is a consequence of (i) and the fact that, due to  Theorem~\ref{thm:equivalence_topologies_nonhybrid}(ii), $\sigma_{\smash{\Theta\widehat\Theta}}$ coincides with $\sigma_D$ under the given assumptions.
\end{proof}
As a consequence of the previous theorems,  one can deduce a first result of continuity of the skew-product semiflow composed of the time translation of the initial vector field and the solutions of the respective delay differential equation.\par\smallskip
Consider $f\in\LC_p(\R^{2N},\R^N)$ and any dense numerable subset $D$ of $\R^N$.
With the notation introduced in Theorem~\ref{thm:13.04-16:49}, let us denote by $\U_{(E,\T)}$ the subset of $\R\times\mathrm{Hull}_{(E,\T)}(f)\times \mathcal{C}$ given by
\begin{equation*}
\U_{(E,\T)}=\bigcup_{g\in\mathrm{Hull}_{(E,\T)}(f)\,,\,\phi\in \mathcal{C}}
\{(t,g,\phi)\mid t\in I_{g,\phi}\}
\end{equation*}
where $(E,\T)\in\{(\LC_p,\T_B), (\LC_p,\T_D), (\LC,\sigma_D)\}$\,.
\begin{thm}
Consider $f\in\LC_p(\R^{2N},\R^N)$ and the map
\begin{equation}\label{eq:skewproductsemiflow}
\begin{split}
 \U_{(E,T)}\subset \R\times\mathrm{Hull}_{(E,\T)}(f)\times \mathcal{C}\;  &\to\; \mathrm{Hull}_{(E,\T)}(f)\times \mathcal{C}\\
 (t,g,\phi) \qquad \qquad &\mapsto \;\; \ \big(g_t, x_t(\cdot,g,\phi)\big)\,.
\end{split}
\end{equation}
\begin{itemize}[leftmargin=20pt]\setlength\itemsep{0.3em}
\item[\rm (i)]
If $(E,\T)=(\LC_p,\T_B)$  the map~\eqref{eq:skewproductsemiflow}
defines a local continuous skew-product semiflow on $\mathrm{Hull}_{(\LC_p,\T_B)}(f)\times \mathcal{C}$.
\item[\rm (ii)] If  $f$ admits $L^p_{loc}$-bounded $l$-bounds and $(E,\T)=(\LC_p,\T_D)$,
the map~\eqref{eq:skewproductsemiflow} defines a local continuous skew-product semiflow on $\mathrm{Hull}_{(\LC_p,\T_D)}(f)\times \mathcal{C}$.
\item[\rm (iii)] If  $f$ has $L^1_{loc}$-equicontinuous $m$-bounds, $L^1_{loc}$-bounded $l$-bounds and $(E,\T)=(\LC,\sigma_D)$, the map~\eqref{eq:skewproductsemiflow} defines a local continuous skew-product semiflow on $\mathrm{Hull}_{(\LC,\sigma_D)}(f)\times \mathcal{C}$.
\end{itemize}
\end{thm}
\begin{proof}
The proof is a direct consequence of \cite[Corollary 3.4]{paper:LNO1}, of Theorem~\ref{thm:ContinuousTBdelay}, Theorem~\ref{thm:ContinuousTDdelay} and Theorem~\ref{thm:ContinuousSigmaDdelay}.
\end{proof}
\subsection{Continuity with respect to $\T_{\Theta B}$,  $\T_{\Theta D}$ and $\sigma_{\Theta D}$}\label{subsec:Ttx}
As follows, we want to look into the new hybrid topologies $\T_{\Theta B}$,  $\T_{\Theta D}$ and $\sigma_{\Theta D}$ presented in Definition~\ref{def:hybridtopologies} and see which advantages they have with respect to $\T_B$, $\T_D$ and $\sigma_D$. Besides the already noticed relations of order given
in~\eqref{eq:DelayTopologiesChain}, we will show how $\T_{\Theta B}$,  $\T_{\Theta D}$ and $\sigma_{\Theta D}$ are the natural choice in the case of delay differential equations. Indeed, the $m$-bounds of the vector field provide valuable information on the modulus of continuity of the solutions so that for the first $N$ components of the space variables, accounting for the present state of the solution, we precisely know on which compact set of continuous function is sufficient to ask for convergence.\par\smallskip
On one hand, this idea allows to restrict the requirement of convergence on bounded sets of continuous functions only to the last $N$ components of the space variables, accounting for the past state of the solution (see $\T_{\Theta B}$ in Definition \ref{def:hybridtopologies}), so that also any possible initial data is covered. On the other hand, if one aims to use some kind of point-wise convergence (either $\T_{\Theta D}$ or $\sigma_{\Theta D}$) the same reasoning permits to obtain the result with assumptions on the sole $l_2$-bounds instead of the whole $l$-bounds of the vector field.
\par\smallskip
As a first step we prove the continuity of the flow defined by the time-translations for such hybrid topologies.
\begin{thm}\label{thm:cont_time_transl_hybrid}
Let $\Theta$ and $\smash{\widehat\Theta}$ be any pair of suitable sets of moduli of continuity, as in {\rm Definition~\ref{def:ssmc}}, and $D$ any dense and countable subset of $\R^N$. The following statements hold.
\begin{itemize}[leftmargin=20pt]\setlength\itemsep{0.3em}
\item[{\rm (i)}] The map
\begin{equation*}
\Pi:\R\times \LC_p(\R^{2N},\R^N)\to \LC_p\big(\R^{2N},\R^N)\, ,\qquad (t,f)\mapsto\Pi(t,f)=f_t\, ,
\end{equation*}
defines a continuous flow on $\left(\LC_p(\R^{2N},\R^N),\T\right)$, with $\T\in\{\T_{\Theta B}, \T_{\Theta \smash{\widehat\Theta}}, \T_{\Theta D}\}$.
\item[\rm(ii)] Moreover, if $E\subset\LC(\R^{2N},\R^N)$ has $L^1_{loc}$-equicontinuous $m$-bounds and it is such that $f\in E$ implies $f_t\in E$ for all $t\in\R$ , then the map
\begin{equation*}
\Pi:\R\times E\to E\, ,\qquad (t,f)\mapsto\Pi(t,f)=f_t\, ,
\end{equation*}
defines a continuous flow on $\left(E,\sigma\right)$, where $\sigma\in\{\sigma_{\Theta \smash{\widehat\Theta}}, \sigma_{\Theta D}\}$.
\end{itemize}
\end{thm}
\begin{proof}
Let us firstly deal with the case $\left(\LC_p(\R^{2N},\R^N), \T_{\Theta B}\right)$.
Given a sequence $(t_n)_\nin$ of real numbers converging to some $t\in\R$ and a sequence $(f_n)_\nin$ converging to some $f$ in $\left(\LC_p(\R^{2N},\R^N), \T_{\Theta B}\right)$, we shall prove that $\left((f_n)_{t_n}\right)_\nin$ converges to $f_t$ in
$\left(\LC_p(\R^{2N},\R^N), \T_{\Theta B}\right)$.  Thanks to Lemma~\ref{lem:alternative-seminorms}, from the convergence of the sequence $(f_n)_\nin$ to $f$ we deduce that
\begin{equation}\label{fnthetaB}
\sup_{x(\cdot)\in\K_j^I,\; u(\cdot)\in C(I,B_j)}\left[\int_I\big|f_n(s,x(s),u(s))-f(s,x(s),u(s))\big|^p ds \right]^{1/p}\xrightarrow{\nti}  0\,.
\end{equation}
 Fixed a function $x(\cdot)\in \K^I_j$ and extending it by constants to the real line $\R$, we can define the functions $\widehat f,\widehat f_n\in\LC_p(\R^{N},\R^N)$, with $n\in\N$, by
\begin{equation*}
\widehat f (s,u)=f\big(s,x(s-t),u\big)\quad\text{and}\quad \widehat f_n (s,u)=f_n\big(s,x(s- t_n),u\big)\,.
\end{equation*}
Notice that $(\widehat f_n)_\nin$ converges to $\widehat f$ in $\big(\LC_p(\R^{N},\R^N),\T_{B}\big)$ because
\begin{align*}
\sup_{u(\cdot)\in C(I,B_j)}&\left[\int_I\big|f_n(s,x(s-t_n),u(s))-f(s,x(s-t),u(s))\big|^p ds \right]^{1/p}\\
 & \!\!\!\!\! \leq \sup_{u(\cdot)\in C(I,B_j)} \left[\int_I\big|f_n(s,x(s-t_n),u(s))-f(s,x(s-t_n),u(s))\big|^p ds \right]^{1/p}\\
&\; + \sup_{u(\cdot)\in C(I,B_j)} \left[\int_I\big|f(s,x(s-t_n),u(s))-f(s,x(s-t),u(s))\big|^p ds \right]^{1/p}\,,
\end{align*}
and the right-hand side goes to zero as $n\to\infty$ because of~\eqref{fnthetaB} and due to the fact that $f\in \LC_p(\R^{2N},\R^N)$ and $(x(\cdot-t_n))_\nin$ converges to $x(\cdot -t)$ on $I$. Therefore, from Theorem~II.1 in~\cite{book:RMGS}  we deduce that $((\widehat f_n)_{t_n})_\nin$ converges to ${\widehat f}_t$ in $\big(\LC_p(\R^{N},\R^N),\T_B\big)$. More precisely, for any interval $I$ with rational extrema and $j\in\N$ one has
\begin{equation}\label{eq:hatfnf}
\sup_{u(\cdot)\in C(I,B_j)}\bigg[\int_I|f_n\big(s+t_n,x(s),u(s)\big)-f\big(s+t,x(s),u(s)\big)|^p\,ds\bigg]^{1/p}\xrightarrow{\nti} 0\,.
\end{equation}
As a matter of fact, \eqref{eq:hatfnf} is also uniform in $x(\cdot)\in \K^I_j$. Indeed, if this were not the case there would exist an $\ep>0$, a subsequence $(t_{n_k})_{k\in\N}$, and a sequence $(x_k)_{k\in\N}$ in $\K_j^I$ such that
\begin{equation}\label{eq:contradiction}
\ep<\sup_{u(\cdot)\in C(I,B_j)}\bigg[\int_I\big|f_{n_k}\big(s+t_{n_k},x_k(s),u(s)\big)-
f\big(s+t,x_k(s),u(s)\big)\big|^p\,ds\bigg]^{1/p}.
\end{equation}
However, if we consider the functions $g,g_k\in\LC_p(\R^{N},\R^N)$, with $k\in\N$, defined  by
\begin{equation*}
g(s,u)=f\big(s,x(s-t ),u\big)\quad\text{and}\quad g_k (s,u)=f_{n_k}\big(s,x_k(s-t_{n_k} ),u\big)\,,
\end{equation*}
where $x(\cdot)$ is the limit, up to a subsequence, of $\big(x_k(\cdot)\big)_{k\in\N}$ (such limit exists because $\K_j^I$ is compact), we have that $(g_{k})_{k\in\N}$ converges to $g$ in $\big(\LC_p(\R^{N},\R^N),\T_B\big)$. Indeed, for any interval $I$ with rational extrema and any $j\in\N$ one has
\begin{align*}
\sup_{u(\cdot)\in C(I,B_j)}&\left[\int_I\big|f_{n_k}(s,x_k(s-t_{n_k}),u(s))-f(s,x(s-t),u(s))\big|^p ds \right]^{1/p}\\
 & \!\!\!\!\!\!\!\!\! \leq \sup_{u(\cdot)\in C(I,B_j)} \left[\int_I\big|f_{n_k}(s,x_k(s-t_{n_k}),u(s))-f(s,x_k(s-t_{n_k}),u(s))\big|^p ds \right]^{1/p}\\
&\; + \sup_{u(\cdot)\in C(I,B_j)} \left[\int_I\big|f(s,x_k(s-t_{n_k}),u(s))-f(s,x(s-t),u(s))\big|^p ds \right]^{1/p}
\end{align*}
and again the right-hand side of the previous inequality goes to zero as $k\to\infty$  because of~\eqref{fnthetaB} and due to the fact that $f\in \LC_p(\R^{2N},\R^N)$ and $(x_k(\cdot-t_{n_k}))_\nin$ converges to $x(\cdot -t)$ on $I$.
Thus, as above, we deduce that  $((g_k)_{t_{n_k}})_\nin$ converges to ${g}_t$ in $\big(\LC_p(\R^{N},\R^N),\T_B\big)$. This fact, together with the convergence of  $(x_k(\cdot))_{k\in\N}$  to $x(\cdot)$ in $I$, contradicts~\eqref{eq:contradiction}, shows that~\eqref{eq:hatfnf} is uniform in $x(\cdot)\in \K^I_j$, and, thanks to Lemma~\ref{lem:alternative-seminorms}, ends the proof of (i) for $\T_{\Theta B}$.
\par\smallskip
The proof for the topologies $ \T_{\Theta \smash{\widehat\Theta}}$ and  $ \T_{\Theta D}$ can be obtained following similar arguments, using Theorem~3.1 and Remark~3.2 in \cite{paper:LNO1} respectively (instead of Theorem~II.1 in~\cite{book:RMGS}). The proof for the weak hybrid topologies $\sigma_{\Theta \smash{\widehat\Theta}}$ and $ \sigma_{\Theta D}$ can be deduced from the one of Theorem 3.1 in \cite{paper:LNO2} with minor modifications.
\end{proof}
The previous result solves the problem of the continuity of the time-translations for the topologies $\T_{\Theta B}$,  $\T_{\Theta D}$ and $\sigma_{\Theta D}$ so that, in order to tackle the continuity of the skew-product semiflow, now we need to address the question of continuous variation of the solutions.
\begin{thm}\label{thm:ContinuousTTBdelay}
Consider $E\in\LC_p(\R^{2N},\R^N)$ with $L^1_{loc}$-equicontinuous $m$-bounds and let $\Theta$ be the suitable set of moduli of continuity given by the $m$-bounds as in {\rm Definition~\ref{def:THETAjseq}}.  Moreover, let $D$ be any dense and countable subset of $\R^N$. With the notation of {\rm Theorem~\ref{thm:13.04-16:49}}, the following statements hold.
\begin{itemize}[leftmargin=20pt]
\item[{\rm (i)}] If $(f_n)_\nin$ is a sequence in $E$ converging to $f$ in $(\LC_p(\R^{2N},\R^N),\T_{\Theta B})$ and $\big(\phi_{n}(\cdot)\big)_\nin$ is a sequence in $\mathcal{C}$ converging uniformly to some function  $\phi\in \mathcal{C}$, then
\begin{equation*}
 x(\cdot,f_n,\phi_{n}) \xrightarrow{\nti}  x(\cdot,f,\phi)
\end{equation*}
uniformly in any $[-1, T]\subset I_{f,\phi}$.
\item[{\rm (ii)}]  If  $E$ has also $L^1_{loc}$-bounded $l_2$-bounds, $(f_n)_\nin$ is a sequence in $E$ converging to $f$ in $(\LC(\R^{2N},\R^N),\sigma_{\Theta D})$ and $\big(\phi_{n}(\cdot)\big)_\nin$ is a sequence in $\mathcal{C}$ converging uniformly to $\phi\in \mathcal{C}$, then
\begin{equation*}
 x(\cdot,f_n,\phi_{n}) \xrightarrow{\nti}  x(\cdot,f,\phi)
\end{equation*}
uniformly in any $[-1, T]\subset I_{f,\phi}$.
\end{itemize}
\end{thm}
\begin{proof}
As regards the proof of (i), proceed as for the proof of Theorem~\ref{thm:ContinuousTBdelay}, except for the fact that instead of \eqref{eq:05.10-13-31}, now one has,
\begin{equation*}
I_1\leq \sup_{\substack{x(\cdot)\in\K^{[0,1]}_k,\\[0.5ex] y(\cdot)\in C([-1,0],B_k)}}\bigg[\int_0^1  \left| f_n\big(t,x(t),y(t-1)\big)-f\big(t,x(t),y(t-1)\big)\right|^pdt\bigg]^{1/p},
\end{equation*}
 and the same reasoning applies, where instead of the topology $\T_B$, now we use the topology $\T_{\Theta B}$.\par\smallskip
On the other hand, (ii) is immediate due to Theorem~\ref{thm:ContinuousSigmaDdelay} and Theorem~\ref{thm:equivalence_topologies}.
\end{proof}
As a consequence of the previous theorems, one can obtain new results of continuity of the induced skew-product semiflow for the topologies $\T_{\Theta B}$, $\T_{\Theta D}$ and $\sigma_{\Theta D}$.
\par\smallskip
Consider $f\in\LC_p(\R^{2N},\R^N)$ and a dense and countable set $D\subset\R^N$.
With the notation of Theorem~\ref{thm:13.04-16:49}, let  $\U_{(E,\T)}$ be the subset of $\R\times\mathrm{Hull}_{(E,\T)}(f)\times \mathcal{C}$ given by
\begin{equation*}
\U_{(E,\T)}=\bigcup_{g\in\mathrm{Hull}_{(\LC_p,\T)}(f)\,,\,\phi\in \mathcal{C}}
\{(t,g,\phi)\mid t\in I_{g,\phi}\}\,,
\end{equation*}
where $(E,\T)\in\{( \LC_p,\T_{\Theta B}),\,(\LC_p,\T_{\Theta D}),\,(\LC,\sigma_{\Theta D})\}$.
\begin{thm}\label{thm:cont-skewprod-hybrid}
Consider $f\in\LC_p(\R^{2N},\R^N)$ with $L^1_{loc}$-equicontinuous $m$-bounds, the suitable set of moduli of continuity  $\Theta$   given by the $m$-bounds, as defined in {\rm Definition~\ref{def:THETAjseq}}, and the map
\begin{equation}\label{eq:skewproductsemiflow2}
\begin{split}
 \U_{(E,T)}\subset \R\times\mathrm{Hull}_{(E,\T)}(f)\times \mathcal{C}\;  &\to\; \mathrm{Hull}_{(E,\T)}(f)\times \mathcal{C}\\
 (t,g,\phi) \qquad \qquad &\mapsto \;\; \ \big(g_t, x_t(\cdot,g,\phi)\big)\,.
\end{split}
\end{equation}
\begin{itemize}[leftmargin=20pt]\setlength\itemsep{0.3em}
\item[\rm(i)]
If $(E,\T)=(\LC_p,\T_{\Theta B})$, the map~\eqref{eq:skewproductsemiflow2}
defines a local continuous skew-product semiflow on $\mathrm{Hull}_{(\LC_p,\T_{\Theta B})}(f)\times \mathcal{C}$.
\item[\rm (ii)] If  $f$ has $L^1_{loc}$-bounded $l_2$-bounds, $(E,\T)\in \{ (\LC,\T_{\Theta D}),\,(\LC,\sigma_{\Theta D})\}$ where $D$ is any contable dense subset of $\R^n$, then the map \eqref{eq:skewproductsemiflow2}
defines a local continuous skew-product semiflow on $\mathrm{Hull}_{(E,\T)}(f)\times \mathcal{C}$.
\end{itemize}
\end{thm}
\begin{proof}
For any of the cited topologies, the continuity on the base flow is a consequence of Theorem~\ref{thm:cont_time_transl_hybrid}. Additionally, Theorem~\ref{thm:ContinuousTTBdelay} allows to deduce the continuity of the solutions for the cases $(\LC_p,\T_{\Theta B})$ and $ (\LC,\sigma_{\Theta D})$ and, in particular, the last one also implies the result for the case $(\LC,\T_{\Theta D})$.
\end{proof}
\section{Continuous dependence of the solutions in the Sobolev spaces $\mathcal{C}^{1,p}$}\label{sec:Sobolev}
One might inquire if under no assumptions on the $l_2$-bounds whatsoever, it is still possible to obtain any result of global continuity for the skew-product semiflows generated by Carath\'eodory differential equations using a topology weaker than $\T_{\Theta B}$. The topologies of type $\T_{\Theta \smash{\widehat\Theta}}$ come in help but at the price of needing more regularity on the initial data and a slightly technical choice of the set $\widehat \Theta$.
\par\smallskip
In particular, as regards the additional regularity on the initial data, it will be sufficient to use Sobolev spaces of continuous functions whose first derivative exists almost everywhere and it is locally integrable. As a matter of fact, such choice is the natural one in the context of Carath\'eodory delay differential equations. Indeed, due to the existence of the $m$-bounds, the solutions of Carath\'eodory delay differential equation intrinsically satisfy such condition as long as $t>0$.
\begin{defn}
Consider $1\le p <\infty$, a compact interval $I\subset \R$, and let $\mathcal{C}^{1,p}(I)$ denote  the vector space defined by
\begin{equation*}
\mathcal{C}^{1,p}(I):=\{x\in\mathcal{C}(I,\R^N)\mid x \text{ differentiable a.e. and }\dot x\in L^p(I)\}.
\end{equation*}
Such space will be endowed with the norm
\[
\|\cdot\|_{\mathcal{C}^{1,p}(I)}\colon\mathcal{C}^{1,p}(I)\to\R^+,\quad x\mapsto\|x\|_{\mathcal{C}^{1,p}(I)}:=\|x\|_{L^\infty(I)}+\|\dot x\|_{L^p(I)}.
\]
In the following, if $I=[-1,0]$ we will simply write $\mathcal{C}^{1,p}$.
\end{defn}
\begin{prop}
$\mathcal{C}^{1,p}(I)$ is a Banach space with respect to the norm $\|\cdot\|_{\mathcal{C}^{1,p}(I)}$.
\end{prop}
Concerning the choice of $\widehat \Theta$, the idea is to associate a set of moduli of continuity to the space $\mathcal{C}^{1,p}$. For every $j\in\N$ consider
\[
\tau_j(h)=\begin{cases}
\displaystyle\sup_{|t-s|<h\,,\; \|\phi\|_{\mathcal{C}^{1,p}}\le j}|\phi(t)-\phi(s)|, &\text{if } 0\le h\le 1\\[2.4ex]
\qquad \quad\tau_j(1) &\text{if } h>1\\
\end{cases}
\]
and notice that for $p>1$, $\tau_j(\cdot)$ is a modulus of continuity shared by all the functions in the ball of radius $j$ in $\mathcal{C}^{1,p}$. Indeed, considered  $\phi\in\mathcal{C}^{1,p}$ with $\|\phi\|_{\mathcal{C}^{1,p}}\le j$, thanks to H\"older inequality we can write
\begin{equation}\label{eq:27.09-19:17}
 |\phi(t+h)-\phi(t)|\le \int_t^{t+h}|\phi'(s)|\,ds\le h^{1/q}\,\|\phi'(\cdot)\|_{L^p([-1,0])}\le j\,h^{1/q}
\end{equation}
where $1/q+1/p=1$, and notice that the last inequality is uniform in $\phi\in\mathcal{C}^{1,p}$ with $\|\phi\|_{\mathcal{C}^{1,p}}\le j$.  Then, for $p>1$ one can consider the suitable set of moduli of continuity defined by
\begin{equation}\label{eq:C1pmoduli}
\widehat \Theta=\left\{\widehat\theta^I_j(\cdot)=\max\{\theta^I_j(\cdot) ,\,\tau_j(\cdot)\}\,\Big| \, j\in\N,\, I=[q_1,q_2], q_1,q_2\in\Q\right\}.
\end{equation}
\begin{thm}\label{thm:C1p>1}
 Consider $p>1$, $E\subset\LC_p(\R^{2N},\R^N)$,  and any dense and countable subset $D$ of $\R^N$. Then, with the notation of {\rm Theorem~\ref{thm:13.04-16:49}}, the following statements~hold.
\begin{itemize}[leftmargin=20pt]\setlength\itemsep{0.3em}
\item[\rm (i)] If $E$ has $L^1_{loc}$-equicontinuous $m$-bounds, let $\Theta$ be the suitable set of moduli of continuity given by the $m$-bounds,  $\widehat\Theta$  the suitable set of moduli of continuity given in \eqref{eq:C1pmoduli}, and  $\overline E$  the closure of $E$  in $\LC_p(\R^{2N},\R^N)$ with respect to the topology $\T_{\Theta \smash{\widehat\Theta}}$. Then, for any sequence $(f_n)_\nin$ in $\overline E$ converging to $f$ in $(\overline E,\T_{\Theta \smash{\widehat\Theta}})$, and  for any sequence $\big(\phi_{n}(\cdot)\big)_\nin$ in $\mathcal{C}^{1,p}$ converging uniformly to $\phi\in \mathcal{C}^{1,p}$, one has that
\begin{equation}\label{convergencec1p}
\| x(\cdot,f_n,\phi_{n})  - x(\cdot,f,\phi)\|_{\mathcal{C}^{1,p}([-1,T])}\xrightarrow{\nti}0
\end{equation}
for any $[-1,T]\subset I_{f,\phi}$.
\item[\rm (ii)]  If $E$ has $L^1_{loc}$-equicontinuous $m$-bounds and $L^p_{loc}$-bounded $l_2$-bounds,  then \eqref{convergencec1p} holds when the closure of $E$, and hence the convergence of the sequence $(f_n)_\nin$  to $f$, is taken with respect to $\T_{\Theta D}$ for the same modulus of continuity $\Theta$ of {\rm(i)}.
\item[\rm (iii)]  If $E$ has $L^p_{loc}$-bounded $l$-bounds, then~\eqref{convergencec1p} holds when the closure of $E$, and hence  the convergence of the sequence $(f_n)_\nin$  to $f$, is taken  with respect to~$\T_D$.
\end{itemize}
\end{thm}
\begin{proof} Firstly notice that, with the notation of {\rm Theorem~\ref{thm:13.04-16:49}}, one has that the solution $x(\cdot, f , \phi)$ of the delay differential problem with vector field $f\in\LC_p(\R^{2N},\R^N)$ and initial data $\phi\in\mathcal{C}^{1,p}$, belongs to $\mathcal{C}^{1,p}(I)$ for any compact interval $I\subset I_{f,\phi}$. Indeed, for any $t\ge 0$ it satisfies $\dot x(t,f,\phi)=f\big(t,x(t),x(t-1)\big)\in L^p_{loc}$. Let us consider any $0\le T< b_{f,\phi}$.
\par\smallskip
Concerning the proof of (i), by construction of $\widehat\Theta$,  one has that for some  $k\in\N$ $\{\phi_n(\cdot)\mid\nin\}\cup\{\phi(\cdot)\}\subset\widehat\K^{[-1,0]}_k$ , and notice that, considered the topology $\sigma_{\Theta \smash{\overline\Theta}}$ constructed as in the statement of Theorem~\ref{thm:ContinuousSigmaDdelay} for $\theta_0(\cdot)=\tau_k(\cdot)$, it is immediate to prove that  $\sigma_{\Theta \smash{\overline\Theta}}\le \T_{\Theta \smash{\widehat\Theta}}$. As a consequence, by Theorem~\ref{thm:ContinuousSigmaDdelay}, one immediately has that the sequence of solutions $(x(\cdot,f_n,\phi_n))_\nin$ converges uniformly  to $x(\cdot,f,\phi)$ in $[-1,T]\subset[-1,b_{f,\phi})$.
\par\smallskip
In order to prove~\eqref{convergencec1p}, we check the continuous variation of the derivatives in $L^p([-1,T])$. To simplify the notation, denote by $x_n(\cdot)=x (\cdot, f_n,\phi_n)$ and by $x(\cdot)=x (\cdot, f,\phi)$. Moreover, consider a compact interval $J\subset [0, b_{f,\phi})$ with rational extrema and such that $[0,T]\subset J$, and $j\geq k$ such that $\{x_n(t)\mid\nin\}\cup\{x(t)\}\subset B_j$ for each $t\in J$. Then one has
\begin{align*}
\big\|\dot x_n (\cdot)-& \dot x(\cdot)\big\|_{L^p(J)}=\big\|f_n\big(\cdot,x_n(\cdot),x_n(\cdot-1)\big)-f\big(\cdot,x(\cdot),x(\cdot-1)\big)\big\|_{L^p(J)}\\[0.5ex]
&\le \big\|f_n\big(\cdot,x_n(\cdot),x_n(\cdot-1)\big)-f\big(\cdot,x_n(\cdot),x_n(\cdot-1)\big)\big\|_{L^p(J)}\\[0.5ex]
&\qquad\qquad+ \big\|f\big(\cdot,x_n(\cdot),x_n(\cdot-1)\big)-f\big(\cdot,x(\cdot),x(\cdot-1)\big)\big\|_{L^p(J)}\\[0.5ex]
&\le \sup_{\xi(\cdot)\in\K^{J}_j,\eta(\cdot)\in\widehat\K^{J-1}_j} \left\|f_n\big(\cdot,\xi(\cdot),\eta(\cdot-1)\big)-f\big(\cdot,\xi(\cdot),\eta(\cdot-1)\big)\right\|_{L^p(J)}\\
&\qquad\qquad+2\, \big\|x_n (\cdot)- x(\cdot)\big\|_{L^\infty(J)}\,\big\|l^{2j}_f(\cdot)\big\|_{L^p(J)},
\end{align*}
and since $f_n\xrightarrow{\T_{\Theta \smash{\widehat\Theta}}}f$, $x_n(\cdot)\xrightarrow{L^\infty(J)}x(\cdot)$, and $l^{2j}_f(\cdot)\in L^p_{loc}$, then the right-hand side of the previous chain of inequalities goes to zero as $\nti$, which eventually implies that  $x_n(\cdot)\xrightarrow{\nti}x(\cdot)$ in $\mathcal{C}^{1,p}([-1,T])$ and concludes the proof of (i).\par\smallskip
Statement (ii) follows from (i) and Theorem~\eqref{thm:equivalence_topologies}(i). In order to prove (iii), from Theorem~\ref{thm:ContinuousTDdelay} we deduce that the sequence of solutions $(x(\cdot,f_n,\phi_n))_\nin$ converges uniformly  to $x(\cdot,f,\phi)$ in $[-1,T]\subset[-1,b_{f,\phi})$. Moreover, from Theorem~\ref{thm:equivalence_topologies_nonhybrid}(i) $(f_n)_{\nin}$ converges to $f$ in $\T_{B}$. Therefore, if in the above chain of inequalities  we change the first term of the last inequality by
\begin{align*}
\sup_{(\xi(\cdot),\eta(\cdot))\in C(J,B_{2j})} \left\|f_n\big(\cdot,\xi(\cdot),\eta(\cdot)\big)-f\big(\cdot,\xi(\cdot),
\eta(\cdot)\big)\right\|_{L^p(J)}\,,
\end{align*}
we deduce that $\big\|\dot x_n(\cdot)-\dot x (\cdot)\big\|_{L^p(J)}\xrightarrow{\nti} 0$, which finishes the proof of (iii).
\end{proof}
In fact, using the same notation and reasoning of Theorem~\ref{thm:C1p>1}, and recalling that $\sigma_{\smash{\Theta\overline\Theta}}\le \T_{\smash{\Theta\overline\Theta}} $ for any pair of suitable sets of moduli of continuity $ \Theta$ and $\overline\Theta$, one can improve the information on the solutions obtained in Theorem~\ref{thm:ContinuousSigmaDdelay} once the interval $[-1,0]$ is disregarded. More precisely, also from the continuous variation of the initial data in $\mathcal{C}([-1,0])$ we obtain  the continuous variation of the solutions in $\mathcal{C}^{1,p}([0,T])$ for any $[0,T]\subset I_{f,\phi}$, as stated in the following proposition whose proof is omitted.
\begin{prop} \label{prop:C}
Consider $p>1$ and a subset $E\subset\LC_p(\R^{2N},\R^N)$. With the notation of {\rm Theorem~\ref{thm:13.04-16:49}}, the following statements hold.
\begin{itemize}[leftmargin=20pt]\setlength\itemsep{0.3em}
\item[\rm (i)]   If $E$ has $L^1_{loc}$-equicontinuous $m$-bounds,  let $\Theta$ be the suitable set of moduli of continuity given by the $m$-bounds. Assume that $\big(\phi_{n}(\cdot)\big)_\nin$ converges  uniformly to $\phi$ in $\mathcal{C}$ and let $\overline\Theta$ be the suitable set of moduli of continuity constructed in {\rm Theorem~\ref{thm:ContinuousSigmaDdelay}(i)}. Therefore,
if $(f_n)_\nin$ is a sequence in $E$ converging to $f$ in $(\LC_p(\R^{2N},\R^N),\T_{\smash{\Theta\overline\Theta}})$ one has that
\begin{equation}\label{convergence[0T]}
\| x(\cdot,f_n,\phi_{n})  - x(\cdot,f,\phi)\|_{\mathcal{C}^{1,p}([0,T])}\xrightarrow{\nti}0
\end{equation}
for any $[0,T]\subset I_{f,\phi}$.
\item[\rm (ii)]   If $E$ has $L^1_{loc}$-equicontinuous $m$-bounds and $L^p_{loc}$-bounded $l_2$-bounds, then \eqref{convergence[0T]} holds when the convergence of the sequence $(f_n)_\nin$  to $f$, is taken with respect to  $\T_{\Theta D}$ for the same modulus of continuity $\Theta$ of {\rm(i)};
\item[\rm (iii)] if $E$ has $L^p_{loc}$-bounded $l$-bounds, then~\eqref{convergence[0T]} also holds when the convergence of the sequence $(f_n)_\nin$  to $f$ is taken  with respect to ~$\T_D$.
\end{itemize}
\end{prop}
As a consequence of \cite[Corollary 3.4]{paper:LNO1},  Theorem~\ref{thm:cont_time_transl_hybrid}(i) and Theorem~\ref{thm:C1p>1}  one immediately obtains the following theorem of continuity of the skew-product semiflow defined on the hull of a function $f\in\LC_p(\R^{2N},\R^N)$
with phase space $\mathcal{C}^{1,p}$. The proof is omitted. \par\smallskip
We denote by
$\U_{(\LC_p,\T)}$ the subset of $\R\times\mathrm{Hull}_{(\LC_p,\T)}(f)\times \mathcal{C}^{1,p}$ given by
\begin{equation*}
\U_{(\LC_p,\T)}=\bigcup_{g\in\mathrm{Hull}_{(\LC_p,\T)}(f)\,,\,\phi\in \mathcal{C}^{1,p}}
\{(t,g,\phi)\mid t\in I_{g,\phi}\}\,,
\end{equation*}
 where $\T$ is any of the topologies defined on $\LC_p(\R^{2N},\R^N)$.
\begin{thm}\label{thm:skew_p>1}
Consider $p>1$, $f\in\LC_p(\R^{2N},\R^N)$ and the map
\begin{equation}\label{eq:skewproductsemiflow3}
\begin{split}
 \U_{(\LC_p,T)}\subset \R\times\mathrm{Hull}_{(\LC_p,\T)}(f)\times  \mathcal{C}^{1,p}\;  &\to\; \mathrm{Hull}_{(\LC_p,\T)}(f)\times  \mathcal{C}^{1,p}\\
 (t,g,\phi) \qquad \qquad &\mapsto \;\; \ \big(g_t, x_t(\cdot,g,\phi)\big)\,.
\end{split}
\end{equation}
\begin{itemize}[leftmargin=20pt]\setlength\itemsep{0.3em}
  \item[\rm (i)] If f has $L^1_{loc}$-equicontinuous $m$-bounds, let  $\Theta$ and $\widehat\Theta$ be the suitable set of moduli of continuity given by the $m$-bounds and~\eqref{eq:C1pmoduli}, respectively.
Then the map~\eqref{eq:skewproductsemiflow3} is a continuous skew-product semiflow for $\T=\T_{\Theta \smash{\widehat\Theta}}$.
\item[\rm (ii)] If f has $L^1_{loc}$-equicontinuous $m$-bounds and  $L^p_{loc}$-bounded $l_2$-bounds, then the map~\eqref{eq:skewproductsemiflow3} defines  a continuous skew-product semiflow for $\T=\T_{\Theta D}$, where $\Theta$ and is the suitable set of moduli of continuity given by the $m$-bounds.
\item[\rm (iii)] If $f$ has  $L^p_{loc}$-bounded $l$-bounds, the map ~\eqref{eq:skewproductsemiflow3} defines  a continuous skew-product semiflow for $\T=\T_D$.
\end{itemize}
\end{thm}
As regards the case $p=1$, an estimate like \eqref{eq:27.09-19:17} is not true anymore, and thus it seems difficult to find a suitable set of moduli of continuity to apply the same reasoning of the proof of Theorem~\ref{thm:C1p>1}. Nevertheless, looking back at Theorem~\ref{thm:ContinuousTTBdelay}, the following result provides a, maybe not surprising, yet significant, refinement of information.
\begin{thm}\label{thm:C11}
Consider $E\subset\LC(\R^{2N},\R^N)$ and any dense and countable subset $D$ of $\R^N$. Then, with the notation of {\rm Theorem~\ref{thm:13.04-16:49}}, the following statements~hold.
\begin{itemize}[leftmargin=20pt]\setlength\itemsep{0.3em}
\item[\rm (i)] If $E$ has $L^1_{loc}$-equicontinuous $m$-bounds and  $L^1_{loc}$-bounded $l_2$-bounds, let $\Theta$ be the suitable set of moduli of continuity given by the $m$-bounds. For any sequence $(f_n)_\nin$ in  $E$ converging to some $f$ in $(\LC(\R^{2N},\R^N),\T_{\Theta D})$  and  any sequence $\big(\phi_{n}(\cdot)\big)_\nin$  in $\mathcal{C}^{1,1}$ converging uniformly to $\phi\in \mathcal{C}^{1,1}$  one has
\begin{equation}\label{eq:convergence[-1T]}
\| x(\cdot,f_n,\phi_{n})  - x(\cdot,f,\phi)\|_{\mathcal{C}^{1,1}([-1,T])}\xrightarrow{\nti}0
\end{equation}
for any $[-1,T]\subset I_{f,\phi}$.
\item[\rm (ii)] If $E$ has  $L^1_{loc}$-bounded $l$-bounds, then~\eqref{eq:convergence[-1T]}holds when the convergence of the sequence $(f_n)_\nin$  to $f$ is with respect to~$\T_D$.
\end{itemize}
\end{thm}
\begin{proof}
Reasoning as in the first part of the proof of Theorem~\ref{thm:C1p>1},  and with the notation of {\rm Theorem~\ref{thm:13.04-16:49}}, one has that  for any $f\in\LC(\R^{2N},\R^N)$ and any $\phi\in\mathcal{C}^{1,1}$, the solution  $x(\cdot, f , \phi)\in\mathcal{C}^{1,1}(I_{f,\phi})$. \par\smallskip
Moreover, notice that $\mathcal{C}^{1,1}\subset \mathcal{C}$, and thanks to Theorem~\ref{thm:ContinuousTTBdelay}, one already has that, for any compact interval  $I\subset I_{f,\phi}$, the sequence of solutions $(x(\cdot,f_n,\phi_n))_\nin$ converges uniformly  to $x(\cdot,f,\phi)$ in $I$. Hence, in order to conclude the proof, we check the continuous variation of the derivatives in $L^1(I)$. Consider a sequence $(f_n)_\nin$  in $E$ converging to $f$ in $(\LC(\R^{2N},\R^N),\T_{\Theta D})$, a sequence $\big(\phi_{n}(\cdot)\big)_\nin$ in $\mathcal{C}^{1,1}$ converging to $\phi\in \mathcal{C}^{1,1}$ with respect to the topology induced by the norm $\|\cdot\|_{\mathcal{C}^{1,p}}$, and a compact interval $[-1,T]\subset[-1,b_{f,\phi})$, with $T\ge0$. Moreover, to simplify the notation, denote by $x_n(\cdot)=x (\cdot, f_n,\phi_n)$ and by $x(\cdot)=x (\cdot, f,\phi)$, and consider a compact interval $J\subset[0,b_{f,\phi})$ with rational extrema and such that $[0,T]\subset J$. Then, one has
\begin{align*}
\big\|\dot x_n (\cdot)& -\dot x(\cdot)\big\|_{L^1(J)}=\big\|f_n\big(\cdot,x_n(\cdot),x_n(\cdot-1)\big)-f\big(\cdot,x(\cdot),x(\cdot-1)\big)\big\|_{L^1(J)}\\[0.5ex]
&\le \big\|f_n\big(\cdot,x_n(\cdot),x_n(\cdot-1)\big)-f_n\big(\cdot,x_n(\cdot),x(\cdot-1)\big)\big\|_{L^1(J)}\\[0.5ex]
&\qquad + \big\|f_n\big(\cdot,x_n(\cdot),x(\cdot-1)\big)-f\big(\cdot,x(\cdot),x(\cdot-1)\big)\big\|_{L^1(J)}\\[0.5ex]
&\le \big\|x_n (\cdot)- x(\cdot)\big\|_{L^\infty(J-1)}\,\big\|l^{2j}_f(\cdot)\big\|_{L^1(J)}+\\[0.5ex]
&\qquad+\sup_{\xi(\cdot)\in\K^{J}_j,\,\eta(\cdot)\in \mathcal{C}(J-1, B_j)}\left\|f_n\big(\cdot,\xi(\cdot),\eta(\cdot-1)\big)-f\big(\cdot,\xi(\cdot),\eta(\cdot-1)\big)\right\|_{L^1(J)}.
\end{align*}
Now, recall that thanks to Proposition~\ref{prop:07.07-19:44p=1} and Theorem~\ref{thm:equivalence_topologies}(i), all the hybrid strong topologies are equivalent on $\mathrm{cls}_{(\LC(\R^{2N},\R^N),\T_{\Theta D})}(E)$, thus in particular the right-hand side of the previous chain of inequalities goes to zero as $\nti$ because $x_n(\cdot)\xrightarrow{L^\infty(J)}x(\cdot)$, $x(\cdot)$ is bounded on the compact interval $J$, and $f_n\xrightarrow{\T_{\Theta B}}f$. As a consequence, one has that  $x_n(\cdot)\xrightarrow{\nti}x(\cdot)$ in $\mathcal{C}^{1,1}([-1,T])$, which concludes the proof of (i). Analogous arguments hold for the case of $E$ with  $L^1_{loc}$-bounded $l$-bounds and allow us to proof (ii).
\end{proof}
As for Theorem~\ref{thm:C1p>1} and Proposition~\ref{prop:C}, and recalling that $\sigma_{\smash{\Theta D}}\le \T_{\smash{\Theta D}} $ for any suitable set of moduli of continuity $ \Theta$ and any $D$ dense and countable subset of $\R^N$,  also for $p=1$ it is possible to refine the result obtained in Theorem~\ref{thm:ContinuousTTBdelay} thanks to Theorem~\ref{thm:C11}. More precisely, under the assumptions of $L^1_{loc}$-equicontinuity of the  $m$-bounds and $L^1_{loc}$-boundedness of the $l_2$-bounds, also  from the continuous variation of the initial data in $\mathcal{C}([-1,0])$, we obtain  the continuous variation of the solutions in $\mathcal{C}^{1,1}([0,T])$ for any $[0,T]\subset I_{f,\phi}$, as stated in the next result whose proof is omitted.
\begin{prop}
With the notation of {\rm Theorem~\ref{thm:13.04-16:49}} the following statements hold.
\begin{itemize}[leftmargin=20pt]\setlength\itemsep{0.3em}
\item[\rm (i)]  If $E$ has $L^1_{loc}$-equicontinuous $m$-bounds and $L^1_{loc}$-bounded $l_2$-bounds, let $\Theta$ be the suitable set of moduli of continuity given by the $m$-bounds.  For any sequence $(f_n)_\nin$ in $E$ converging to $f$ in $(\LC(\R^{2N},\R^N),\T_{\smash{\Theta D}})$ and any sequence $\big(\phi_{n}(\cdot)\big)_\nin$ in $\mathcal{C}$ converging uniformly to $\phi\in \mathcal{C}$, one has that
\begin{equation}\label{eq:convergence[0T]}
\| x(\cdot,f_n,\phi_{n})  - x(\cdot,f,\phi)\|_{\mathcal{C}^{1,1}([0,T])}\xrightarrow{\nti}0
\end{equation}
for any $[0,T]\subset I_{f,\phi}$.
\item[\rm (ii)]   If $E$ has $L^1_{loc}$-bounded $l$-bounds, then~\eqref{eq:convergence[0T]} also holds when the convergence of the sequence
    $(f_n)_\nin$  to $f$ is  with respect to~$\T_D$.
\end{itemize}
\end{prop}
Finally, one obtains the following theorem for the continuity of the skew-product semiflow on the hull of a function $f\in\LC(\R^{2N},\R^N)$  with phase space $\mathcal{C}^{1,1}$.\par\smallskip
We denote by
$\U_{(\LC,\T)}$ the subset of $\R\times\mathrm{Hull}_{(\LC,\T)}(f)\times \mathcal{C}^{1,1}$ given by
\begin{equation*}
\U_{(\LC,\T)}=\bigcup_{g\in\mathrm{Hull}_{(\LC,\T)}(f)\,,\,\phi\in \mathcal{C}^{1,1}}
\{(t,g,\phi)\mid t\in I_{g,\phi}\}\,.
\end{equation*}
 where $\T$ is any of the topologies defined on $\LC(\R^{2N},\R^N)$.
\begin{thm} Consider $f\in\LC(\R^{2N},\R^N)$, any dense and countable subset $D$ of $\R^N$, and  the map
\begin{equation}\label{eq:skew-p.continuity-p=1}
\begin{split}
\Phi:\U\subset\R^+\times \mathrm{Hull}_{(\LC,\T)}(f)\times \mathcal{C}^{1,1}\quad&\to\quad\mathrm{Hull}_{(\LC,\T)}(f)\times \mathcal{C}^{1,1},\\
\ (t,f,\phi)\qquad \qquad \qquad &\mapsto\quad\qquad \big(f_t, x_t(\cdot,f,\phi)\big)\,.
\end{split}
\end{equation}
\begin{itemize}[leftmargin=20pt]\setlength\itemsep{0.3em}
\item[\rm (i)] If $f\in\LC(\R^{2N},\R^N)$  has $L^1_{loc}$-equicontinuous $m$-bounds and $L^1_{loc}$-bounded $l_2$-bounds,  let  $\Theta$ be the suitable set of moduli of continuity given by the $m$-bounds. Then the map~\eqref{eq:skew-p.continuity-p=1} defines a continuous skew-product semiflow for $\T=\T_{\Theta D}$.
\item[\rm (ii)]  If $f$ has  $L^1_{loc}$-bounded $l$-bounds, then the map~\eqref{eq:skew-p.continuity-p=1}  defines a continuous skew-product semiflow for $\T=\T_D$.
\end{itemize}
\end{thm}
\begin{proof}
Reasoning as in the first part of the proof of Theorem~\ref{thm:C1p>1} one has that $\Phi$ is well-defined. The continuity in the two components is a direct consequence of Theorem~\ref{thm:cont_time_transl_hybrid} and Theorem~\ref{thm:C11} for statement (i), and \cite[Corollary 3.4]{paper:LNO1} and Theorem~\ref{thm:C11} for (ii).
\end{proof}
\begin{rmk}
Again, as for the previous section, when a result of equivalence of the topologies (either Theorem~\ref{thm:equivalence_topologies_nonhybrid} or Theorem~\ref{thm:equivalence_topologies}) holds on $E$, we chose to write the continuity of the  skew-product semiflow using the most practical available topology, i.e. the one that involves some kind of point-wise convergence. It is implicit that the same result holds with respect to any of the other equivalent or stronger topologies of type~\eqref{eq:DelayTopologiesChain}.
\end{rmk}

\end{document}